\documentclass[reqno]{amsart}

\usepackage{amsmath}
\usepackage{multirow}
\usepackage{float}
\allowdisplaybreaks[2]

\usepackage{graphicx}
\usepackage{latexsym}
\usepackage{bbm}
\usepackage{color}
\usepackage[misc]{ifsym}                     
\usepackage{lipsum}
\usepackage{amsmath,amssymb,amsfonts}
\usepackage{subfigure}
\usepackage{array}
\usepackage{mathrsfs}
\usepackage{mathptmx}
\usepackage{empheq}
\usepackage{amsfonts}
\usepackage{amssymb}
\usepackage{dsfont}
\usepackage{amsmath}
\usepackage{amsthm}
\usepackage{grffile}
\usepackage{subfigure}
\usepackage[justification=centering]{caption}
\usepackage{color}
\usepackage{array}
\usepackage{mathrsfs}
\usepackage{multirow}
\usepackage{latexsym}
\usepackage{bbm}
\usepackage{booktabs}

\usepackage{mathptmx}
\usepackage{graphicx}
\usepackage{empheq}
\usepackage{caption}

\usepackage[foot]{amsaddr}

\newtheorem{Theo}{Theorem}
\newtheorem{Assu}{Assumption}
\newtheorem{Rem}{Remark}
\newtheorem{Lem}{Lemma}
\newtheorem{Pro}{Proposition}

\newtheorem{Exa}{Example}
\numberwithin{equation}{section}
 \numberwithin{Lem}{section}
 \numberwithin{Defi}{section}
 \numberwithin{Theo}{section}
 \numberwithin{Pro}{section}
 \numberwithin{Rem}{section}
  \numberwithin{Coro}{section}
  \numberwithin{Fig}{section}
 \numberwithin{Exa}{section}
 \def\RR{\mathbb{R}}

\def\EE{\mathbb{E}}

\def\RR{\mathbb{R}}

\def\EE{\mathbb{E}}
\def\sR{\mathsf{R}}
\def\qqquad{\qquad\quad}

\newcommand{\blc}{\big(}
\newcommand{\brc}{\big)}

\newcommand{\lc}{\left(}
\newcommand{\rc}{\right)}
\newcommand{\lk}{\left[}
\newcommand{\rk}{\right]}
\newcommand{\lft}{\left}
\newcommand{\rht}{\right}

\begin{document}

\title[Mean square stability of STM for SDEs driven by fBm]{Mean square stability of stochastic theta method for stochastic  differential equations driven by fractional Brownian motion}

%
\author{Min Li $^1$}
\address{$^1$School of Mathematics and Physics \emph{and} Center for Mathematical Sciences, China University of Geosciences, Wuhan, China}
\email{liminmaths@163.com}

\author{Yaozhong Hu $^2$}
\address{$^2$Department of Mathematical and Statistical Sciences, University of Alberta, Edmonton, Canada}
\email{yaozhong@ualberta.ca}

\author{Chengming Huang $^3$}
\address{$^3$School of Mathematics and Statistics \emph{and} Hubei Key Laboratory of Engineering Modelling and Scientific Computing, Huazhong University of Science and Technology, Wuhan, China}
\email{chengming\_huang@hotmail.com}

\author{Xiong Wang $^{2,\ast}$}
\email{xiongwang@ualberta.ca}

\address{$^{\ast}$Corresponding author}

%
%
%
%
%
%

\begin{abstract}
 In this paper,   we study the mean square stability of
the solution and its stochastic theta scheme
  for the following stochastic differential equations
 drive   by fractional Brownian motion with Hurst parameter $H\in (\frac 12,1)$:
 $$
 	dX(t)=f(t,X(t))dt+g(t,X(t))dB^{H}(t).
 $$
Firstly, we consider the special case when $f(t,X)=-\lambda\kappa t^{\kappa-1}X$ and $g(t,X)=\mu X$. The solution is explicit and is mean square stable when  $\kappa\geq 2H$.
It is proved that if the parameter   $2H\leq\kappa\le 3/2$ and $ \frac{\sqrt{3/2}\cdot e}{\sqrt{3/2}\cdot e+1}
(\approx 0.77)\leq \theta\leq 1$
  or $\kappa>3/2$ and $1/2<\theta\le 1$,
  the stochastic theta method   reproduces  the mean square stability;
and that  if $0<\theta<\frac 12$, the numerical method does  not  preserve this stability unconditionally.
Secondly, we study the stability of the solution and its stochastic theta scheme  for   nonlinear equations. Due to presence of long memory, 
even the problem of stability   in the mean square sense  of the solution   has not been    well studied 
since the  conventional techniques powerful for stochastic differential equations driven by Brownian motion are no longer   applicable.  Let alone the
stability of  numerical schemes.   We need to develop   completely  new set of  techniques to deal with this difficulty. 
Numerical examples are carried out to illustrate our theoretical results. 
\end{abstract}




\maketitle

\section{Introduction  and main results}\label{sec1}
{ Numerical stability analysis of stochastic differential equations (SDEs) 
 is an important topic in numerical analysis and scientific computing.}
  In order to get insight into the stability behavior of numerical methods
  for SDEs, Saito and Mitsui 
   \cite{Mitsui1996stability}
 studied the mean square stability of several  numerical  schemes  for
  the following stochastic test problem driven by standard Brownian motion (Bm)
 \begin{align}\label{eq.SDE_B}
 dX(t)=\lambda X(t)dt+\mu X(t)dB(t), ~~~\lambda, \mu\in \mathbb{C}\,,
 \end{align}
 with initial value $X(0)\neq 0$  with probability 1 and $\mathbb{E}\mid X(0)\mid ^2<\infty$,  where $dB(t)$ is interpreted in It\^o sense.
 The solution of \eqref{eq.SDE_B} is said to be \emph{mean square stable} if
 \begin{equation}\label{defn.MSS}
 	\lim_{t\to \infty}\mathbb{E}\mid X(t)\mid ^2=0\,.	
 \end{equation}
 As is well-known, the mean square stability of \eqref{eq.SDE_B} is characterized by
\begin{equation*}
 Re(\lambda)+\frac{1}{2}\lvert \mu\rvert^2 <0,
\end{equation*}
where $Re (\lambda)$ denotes the real part of $\lambda$.  
Higham \cite{Higham2000A-stability,Higham2000stability} studied the mean square stability properties
 of stochastic theta method and stochastic theta Milstein method for the  test  equation  \eqref{eq.SDE_B}.
The A-stability (which means that the numerical method preserves the stability of the
underlying test problem unconditionally) of stochastic theta method (STM) and the stochastic theta Milstein method
is  proved
when   $\theta\geq \frac{1}{2}$ and $\theta\geq \frac{3}{2}$, respectively.

Subsequently,  the   stability of the numerical method for nonlinear SDEs driven by Brownian motion
\begin{equation}\label{gene_SDE}
	dX(t)=f(t,X(t))dt+g(t,X(t))dB(t)
\end{equation}
  received much attention in the past decades. Assume that the drift coefficient $f$ satisfies certain  monotone condition,
  and the diffusion coefficient satisfies the linear growth condition.
 The authors in \cite{Higham2003Exponential,Schurz2001}
 proved that the backward Euler method and the split-step backward Euler method
  reproduce the exponential mean square stability of the underlying nonlinear problem.
 More recently, some scholars studied nonlinear stability under a coupled condition on the
 drift and diffusion coefficients. This condition allows that the diffusion coefficient
 grows super-linearly. For example, Szpruch and Mao \cite{szpruch2010strong}
 studied the asymptotic stability in this nonlinear
  setting for the STM. Huang
   \cite{Huang2014SDDEs} proved that for all given step size $\Delta t>0$,
    the STM with $\theta\in [1/2, 1]$ is mean square stable for stochastic delay differential equations
    under the following coupled condition:
  \begin{equation}\label{Cond_Nonl}
  	u^{T}f(t,u,v)+\frac 12 \mid g(t,u,v)\mid ^2\leq \widetilde{\alpha} \mid u\mid ^2+ \widetilde{\beta} \mid v\mid ^2\,,	\qquad \forall t>0,~~u, v\in\RR\,,
  \end{equation}
   with $\widetilde{\alpha}+\widetilde{\beta}<0$. If there exist positive constant $K_1$ and $K_2$ such that
   the drift coefficients $f$ also satisfy
   \begin{equation}\label{Cond_Non2}
   \mid f(t,u,v)\mid \leq K_1 \mid u\mid ^2 + K_2 \mid v\mid ^2,
   \end{equation}
   then the STM with $\theta\in [0, 1/2]$ is mean square stable under certain  stepsize constraint.
    For more details for nonlinear stability of numerical method for SDEs we refer to
   \cite{higham2021introduction} and  references therein. 

In recent decades   long   memory processes have been widely studied and applied  by mathematicians and statisticians. In particular,  the theory of  stochastic differential equations driven by fractional Brownian motions have been
well-developed and have found   applications in various fields (e.g. \cite{BHOZ08, Mishura2008book}).
 For example,  the thermal   dynamics   characterized by a fractional
 Ornstein-Uhlenbeck process
 based on  empirical observation in \cite{Weather2002} is applied in the pricing of weather derivatives;  The  arbitrage  in the financial market  is eliminated in the case of geometric fBm in
 \cite{huoksendal03,Guasoni2006FBM} and in the case of fractional Langevin equation in \cite{Guo2013Flangevin}. The readers can also find interesting applications of fBm in modeling anisotropic multidimensional data with self-similarity and long-range dependence in \cite{WD2020} and references therein. 

Motivated by the above works, we are  concerned   with the  mean square stability
 analysis of
\emph{stochastic theta method} for   some  \emph{stochastic test equations }
driven by fractional Brownian motion (fBm) in $\mathbb{R}^{d}$.
We   focus our effort on the stability problem of the numerical scheme and
 try to avoid the complicate issues  of    existence and uniqueness of the solution when $H<1/2$.
  For this reason  we shall assume exclusively $H>1/2$ throughout the paper.  We also assume $d=1$.
First, a  natural choice of the test equation is the extension of \eqref{eq.SDE_B}, namely, we replace the Brownian motion in \eqref{eq.SDE_B} by fractional Brownian motion.
However, an easy computation (similar to  the one shown below) immediately gives that for any parameters $\lambda$ and $\mu$, the solution to $dX(t)= -\lambda   X(t)dt+\mu X(t)dB^{H}(t)$
with a nonzero initial condition $X(0)=x\in \RR\backslash\{0\}$ will  never be  stable in the mean square sense
(or any $L_p$ sense for any finite $p$). So,  the first thing we shall do is to modify \eqref{eq.SDE_B} to the following  new type of     test equations:
\begin{align}\label{eq.SDE_FB}
dX(t)=-\lambda\kappa t^{\kappa-1}  X(t)dt+\mu X(t)dB^{H}(t)\,,\quad t\ge 0\,, ~~~X(0)=X_0\,,
\end{align}
with $\lambda,\mu\in \mathbb{R}$ and  $\kappa\geq 2H$.   {Here, for simplicity  we assume that $X_0$ is a non-zero constant}.  Notice that \eqref{eq.SDE_FB} has an additional   factor $t^{\kappa-1}$ than \eqref{eq.SDE_B} in the drift term.
 By  the chain rule formula (e.g. \cite[Proposition 2.7]{Hu13} or
 \cite[Lemma 2.7.1]{Mishura2008book}), we have $X(t)=X_0\exp(-\lambda t^{\kappa}+\mu B^H(t))$ and hence
\begin{equation}\label{eq.L2_FB}
	\mathbb{E}\mid X(t)\mid ^{2}=\mathbb{E}(X_0)^2\exp\left[2(-\lambda t^{\kappa}+\mu^2t^{2H})\right]\,.
\end{equation}
 This  formula implies  the mean
square stability of the  solution to \eqref{eq.SDE_FB} if
\begin{equation}\label{Cond_MS}
	 {\text{(i)}\ \kappa>2H \text{~and~} \lambda>0 \quad\text{~or~}\quad  \text{(ii)}\ \kappa=2H \text{~and~} -\lambda+\mu^2<0\,.}
\end{equation}
Otherwise, the solution of \eqref{eq.SDE_FB} diverges   in mean square sense as $t$ goes to infinity.   So we only need to consider \eqref{eq.SDE_FB} for the above two parameter regions \eqref{Cond_MS}.

After we obtain the stability result for the above linear equations \eqref{eq.SDE_FB},   we shall next focus our effort on the numerical stability of the STM for the following \emph{nonlinear SDEs} which are  long memory version  of \eqref{gene_SDE}
\begin{align}\label{gene_FSDE}
dX(t)=f(t,X(t))dt+g(t,X(t))dB^{H}(t)\,,
\end{align}
where $B^{H}(t)$ is  fractional Brownian motion (fBm) with Hurst parameter $H> 1/2$.
Inspired by the conditions \eqref{Cond_Nonl}, \eqref{Cond_Non2} and \eqref{Cond_MS}, we shall assume the coefficients in the SDE \eqref{gene_FSDE} satisfy the following conditions (we assume $d=1$):
\begin{Assu}\label{gene-assum}
There exist constants $\kappa\geq 2H$, $\lambda>0$, $\bar{\lambda}>0$ and $\mu>0$ such that for any $t>0$ and $x\in\RR$
\begin{align}
\textbf{Monotone condition}: \qquad &xf(t,x)\leq -\lambda \kappa t^{\kappa-1} x^2\,,   \label{Lip_f}\\
\textbf{Linear growth}: \qquad &\mid f(t,x)\mid \leq \bar{\lambda} \kappa t^{\kappa-1} \mid x\mid \,, \label{LG_f}\\
\textbf{Uniform linear growth}: \qquad &\mid g(t,x)\mid \leq \mu \mid x \mid \,.  \label{LG_g}
\end{align}
\end{Assu}

\begin{Rem}
We mention that when $\kappa=1$, conditions \eqref{Lip_f} and \eqref{LG_f} reduce to the classical
 monotone condition and linear growth condition (which is discussed in the Brownian motion  case, e.g. \cite{Higham2003Exponential}).
\end{Rem}

 We mention that
 it  is a difficult problem 
  to give the long-time stability analysis of the
    solution of \eqref{gene_FSDE} under Assumption  \ref{gene-assum}.
  there are only  few results about  the moment bounds  of the solution $X(t)$. For example,  the moment bounds is given in   \cite{HN2007}  when $f(t,X)=0$ and $g(t,X)=\sigma(X)$. {More recently, Fan and Zhang \cite{FZ2021}   obtained the moment bounds  with irregular  drift term.
 We shall  show that  under the condition \eqref{Lip_f} and \eqref{LG_f}  and when  $g(t,X(t))=c(t)X(t)$,  the solution $X(t)$ of \eqref{gene_FSDE} is mean square stable. 
 \begin{Rem}We believe that this stability  result is new.
  If we there is a bijection function $h(t,\cdot) :\RR\to \RR$ 
  such that $ g(t,x) \frac{\partial }{\partial x} h(t, x)= c(t) h(t, x)$ and $Y_t=h(t, X_t)$. Then the chain rule formula  yields 
\begin{eqnarray*} 
 dY_t&=& \frac{\partial }{\partial x} h(t, X_t) f(t, X_t)dt+
 \frac{\partial }{\partial x} h(t, X_t) g(t, X_t)dB^H(t)\\
 &=& \tilde f(t,Y_t)dt+c(t) Y(t) dB^H(t)\,,  
  \end{eqnarray*} 
  where $\tilde f(t,y)=(\frac{\partial }{\partial x} h)(t, h^{-1}(t,y)) f(t, h^{-1}(t,y))$  with $h^{-1}(t,x)$ denoting the inverse function
  $h(t,\cdot) :\RR\to \RR$.  So, we may sometimes reduce the general equation to our case. But we are not pursuing this direction in current  work. 
  \end{Rem} 
 On the other hand,  from the numerical viewpoint, we investigate the stability of the STM for the general equation \eqref{gene_FSDE}
 rather than \eqref{Sim_SDE} based on Assumption \ref{gene-assum}. We hope the numerical results would  also 
 provide some insights for the theoretical stability  analysis of the solution to \eqref{gene_FSDE}. 

The numerical scheme that we propose  to study is the stochastic theta method (STM) to \eqref{gene_FSDE},
which is some kind of implicit-explicit $\theta$ Euler-Maruyama scheme:
\begin{empheq}[left=(\text{STM})\empheqlbrace]{align}\label{STM}
X_{n+1}&=X_{n}+ \theta f(t_{n+1},X_{n+1})\Delta t
+(1-\theta) f(t_n,X_n)\Delta t + g(t_n,X_n)V_{n}^H, \\
&\hbox{where \quad  $t_{n}=n\cdot \Delta t$ \quad and \quad   $V_{n}^{H}=B^{H}(t_{n+1})-B^{H}(t_{n})$,  }\nonumber \\
&\qquad \qquad \hbox{$n=0, 1, 2, \cdots$ \quad and \quad $\Delta t>0$ is fixed stepsize.}
\nonumber
\end{empheq}
In particular, when $f(t,X)=-\lambda\kappa t^{\kappa-1}X$ and $g(t,x)=\mu X$, \eqref{STM} becomes
\begin{equation}\label{STM_linear}
	X_{n+1}=X_{n}-\kappa\lambda \theta\cdot (t_{n+1})^{\kappa-1}X_{n+1}\Delta t-\kappa\lambda (1-\theta)\cdot (t_n)^{\kappa-1} X_n \Delta t+\mu\cdot X_n V_{n}^{H},
\end{equation}
The main stability theorems we shall prove are displayed as follows:
\begin{Theo}\label{Thm_1}
Let $\Delta t>0$ be fixed and let {$\lambda$, $\mu$ satisfy  \eqref{Cond_MS}}. For the test equation \eqref{eq.SDE_FB} and the STM \eqref{STM_linear}  we have the following  statements.
\begin{enumerate}
\item[(i)] If $\kappa\geq 2H$ and $\frac{\sqrt{3/2}\cdot e}{\sqrt{3/2}\cdot e+1}\leq \theta\leq 1$, then
the STM \eqref{STM_linear} is  mean square stable for the   test equation  \eqref{eq.SDE_FB}, namely, $\displaystyle \lim_{n\to \infty} \EE \mid X_n\mid ^2=0$.
\item[(ii)] If $\kappa> 3/2$ and $\frac{1}{2}< \theta\leq 1$, then
the STM \eqref{STM_linear} is  mean square stable for the   test equation  \eqref{eq.SDE_FB}.
\item[(iii)] If $\kappa\geq 2H$ and $0<\theta<\frac 1 2$, then the STM \eqref{STM_linear} is \emph{not unconditionally} mean square stable  for  the test equation  \eqref{eq.SDE_FB}.
 \end{enumerate}
\end{Theo}
\begin{Rem}
We are not clear whether or not
the STM \eqref{STM_linear} is  mean square stable when   $2H\leq \kappa\leq \frac 32$ and $\frac 12\leq \theta<\frac{\sqrt{3/2}\cdot e}{\sqrt{3/2}\cdot e+1}$,   which will be a topic for  future research.
\end{Rem}

\begin{Theo}\label{Thm_2}
Let $\Delta t>0$ be fixed and let {$\lambda$, $\mu$ in Assumption \ref{gene-assum} satisfy  \eqref{Cond_MS}}. 
 For the SDEs with fBm \eqref{gene_FSDE} and 
 the STM \eqref{STM} we have the following statement. 
\begin{enumerate}
\item[(i)] If  \eqref{Lip_f} and \eqref{LG_g} in Assumption \ref{gene-assum} hold, then
the STM \eqref{STM} with $\theta=1$ (i.e., the backward Euler method) is  mean square stable for
 the  equation \eqref{gene_FSDE}. 
\item[(ii)] If \eqref{Lip_f},  \eqref{LG_f} and \eqref{LG_g}  in Assumption \ref{gene-assum} hold, 
then the STM \eqref{STM} with $\frac{\sqrt{6}e\bar{\lambda}/\lambda}{\sqrt{6}e\bar{\lambda}/\lambda+1}\leq \theta <1$ 
  is  mean square stable for the  equation \eqref{gene_FSDE}. 
\end{enumerate}
\end{Theo}

We shall prove  Theorems \ref{Thm_1} and  \ref{Thm_2} in Sections 2 and   3, respectively.
Before we end this section we would  point out the new  difficulties we encounter  compared with the classical
Brownian motion (e.g. see subsection 2.5).  We can write  \eqref{STM_linear}  as
\begin{equation}
X_{n+1} =\,\left(\frac{1-\kappa(1-\theta)\lambda (t_n)^{\kappa-1} \Delta t}{1+\kappa\theta\lambda (t_{n+1})^{\kappa-1}\Delta t}
+\frac{\mu V_n^{H}}{1+\kappa\theta\lambda (t_{n+1})^{\kappa-1}\Delta t}\right)X_{n} \,.
\end{equation}

%
When $H=1/2$ (i.e. the Brownian motion case), $X_{n+1}$ is the product of independent variables and the corresponding computation is  much easier.
 However, this is no longer true    in our fBm setting.
 We   encounter two major  difficulties:
 \begin{enumerate}
   \item The increments   $B^{H}(t_{n+1})-B^{H}(t_{n})$ of the fractional Brownian motion 
 depend  on the past history, which makes the stability analysis much more sophisticated.
   \item The fractional Brownian motion lacks   martingale property or Markov property so that some useful techniques such as conditional expectation seems impossible or at least over-sophisticated.  
 \end{enumerate}
To get around these difficulties we shall  employ some other analysis and  computation techniques.
In fact, in the proof of different parts of Theorem \ref{Thm_1}, we 
shall use different techniques.  For example, 
in the proof of part (i) of Theorem \ref{Thm_1}  we use  the technique
  of   generalized polarization, raw moments formula of Gaussian distributions and the asymptotic properties of
 confluent hypergeometric function.  On the other hand, 
 the main tool to prove   part (ii) is the celebrated Gaussian
 correlation inequality. Finally, the statement  of part (iii) is proved through  the
 strong law of large numbers of dependent random variables.
All of these are  done in   Section $2$.
Let us mention that the test equation \eqref{eq.SDE_FB} has not been previously studied even  when the fBm is replaced by the standard Brownian motion
  and it is interesting to carry out  the stability  analysis of  the corresponding stochastic theta scheme for its own sake and for the comparison purpose.
   This is  also done in   Section $2$.  
The   proof of Theorem \ref{Thm_2} is analogous to that of  part (i) of Theorem \ref{Thm_1} and is  provided in Section $3$. 
  In Section $4$,  some  numerical simulations are presented to validate  our
theoretical results. Finally,  some concluding remarks are given in the last section.

\section{STM: Mean square linear stability analysis}
In this section we shall prove our main result, i.e.,  Theorem \ref{Thm_1}.
The parts (i), (ii) and (iii) are  proved in subsection \ref{subsection 2.2}, \ref{subsection 2.3} and \ref{subsection 2.4}, respectively.

 Obviously,
\eqref{STM_linear} is equivalent to the following recurrent equation 
\begin{align}
X_{n+1}&= \left(\alpha_{n}(\theta,\lambda,\Delta t)+\beta_{n}(\theta,\lambda,\mu,\Delta t)V_{n}^{H}\right) X_{n}\,, \nonumber
\end{align}
where $\kappa\geq 2H>1$ and
\begin{empheq}[left=\empheqlbrace]{align}
&\alpha_{n}(\theta,\lambda,\Delta t)=\, \frac{1-\kappa(1-\theta)\lambda (t_n)^{\kappa-1} \Delta t}{1+\kappa\theta\lambda (t_{n+1})^{\kappa-1}\Delta t}=\frac{1-\kappa(1-\theta)\lambda n^{\kappa-1} \Delta t^{\kappa}}{1+\kappa\theta\lambda (n+1)^{\kappa-1}\Delta t^{\kappa}}\,, \label{e.2.1}\\
&\beta_{n}(\theta,\lambda,\mu,\Delta t)=\, \frac{\mu}{1+\kappa\theta\lambda (t_{n+1})^{\kappa-1}\Delta t}=\frac{\mu}{1+\kappa\theta\lambda (n+1)^{\kappa-1}\Delta t^{\kappa}}\,.
 \label{e.2.2}
 \end{empheq}
For notational simplicity, throughout the remaining part of the paper we denote $\alpha_{n}(\theta,\lambda,\Delta t)$,
$\beta_{n}(\theta,\lambda,\mu,\Delta t)$ by $\alpha_{n}$ and $\beta_{n}$, respectively.  {Note that \eqref{e.2.1} and \eqref{e.2.2} are well defined if we require the condition \eqref{Cond_MS} or otherwise the denominators in the expressions of  $\alpha_{n}$ and $\beta_{n}$ could be $0$.}

\subsection{Heuristic arguments}
  Before the   proof, we would like to   explain why Theorem \ref{Thm_1} could  hold true heuristically, namely,  why   the STM
  \eqref{STM_linear} is    stable  when $\theta>1/2$ and is unstable when $\theta<1/2$, formally.  Denote
\begin{align}
Z_{n}(\Delta t)=\alpha_{n}+\beta_{n}V_{n}^{H}\,.\label{e.2.3}
\end{align}
Then we have
\begin{align}
X_{n+1}=X_{0}\prod_{k=0}^{n}Z_{k}(\Delta t)=X_{0}\prod_{k=0}^{n}\left(\alpha_{k}
+\beta_{k}V_{k}^{H}\right)\,. \label{q4}
\end{align}
 Obviously, for fixed $\Delta t$, $\lambda$ and $\mu$,
$$
\lim_{n\to \infty}\alpha_{n}=-\frac{1-\theta}{\theta},~~~~~~
\lim_{n\to \infty}\beta_{n}=0.
$$
Notice that this is quite different than the setting with $H=1/2$ where $\alpha_n$ and $\beta_n$ do not depend on $n$ because of the absence of $(t_n)^{\kappa-1}$ for $\kappa=2H=1$ (see Section 3 for more details). Formally,   if we could  think   $\{V_{k}^{H}\}$ in \eqref{q4} as a sequence of finite numbers,  then by the limits of $\alpha_n$ and $\beta_n$,
we would have
\begin{equation*}
	\mid X_{n+1}\mid ^2=\mid X_{0}\mid ^2\prod_{k=0}^{n}\left(\alpha_{k}
+\beta_{k}V_{k}^{H}\right)^2 \asymp  \left(\frac{1-\theta}{\theta}\right)^{2n}\to\begin{cases}
		0\,, &\text{if~} \frac{1}{2}<\theta\leq 1\,;\\
		\infty\,, &\text{if~} 0\leq \theta< \frac 1 2\,,
	\end{cases}
\end{equation*}
where and through the remaining part of this paper, we use
$a_n\asymp b_n$ to denote that there are two positive constants $c_1$ and $c_2$, independent of  $n$,  such that $c_1a_n\le b_n\le c_2 a_n$  for all $n\ge 1$.

However, the random variables  $\{V_{k}^{H}\}$ in our setting are not uniformly bounded. Even worse, they are long range dependent. Therefore, the  above heuristic  argument   cannot  be applied directly to analyze  \eqref{q4}, especially for the scenario of (mean square) stability.  Presumably,   there are two ways to break   these barriers.
\begin{enumerate}
	\item[(1)] Choose $\theta$ carefully so that the oscillation caused by $\{V_{k}^{H}\}$ can still be manageable.
	\item[(2)] Take $\kappa$   sufficiently large so that $\beta_n \cdot V_{k}^{H}$ converges to $0$ fast enough so that  influences of $\{V_{k}^{H}\}$ can be neglected.
\end{enumerate}
Our proof will follow these spirits but with much more sophisticated  tricks and computations. For example, we need to use the asymptotics of the confluent hypergeometric functions which comes from the moments of Gaussian variables.


\subsection{The case of $\kappa\geq 2H$ and $\frac{\sqrt{3/2}\cdot e}{\sqrt{3/2}\cdot e+1}\leq \theta\leq 1$} \label{subsection 2.2}
In this subsection we prove part (i) of the main theorem, namely, we consider the case
when  $\kappa\geq 2H$ and $\frac{\sqrt{3/2}\cdot e}{\sqrt{3/2}\cdot e+1}\leq \theta\leq 1$. Firstly, we state a  useful lemma, which  is a generalization of polarization identity.
\begin{Lem}{\cite[Lemma 1]{Kan2008}}\label{Polarization}
	Let $x_1,\dots,x_n$ be real numbers, and let $s_1,\dots,s_n$ be nonnegative integers and $s=\sum_{i=1}^n s_i$. Then, we have
	\[
	 x_1^{s_1}\cdots x_n^{s_n}=\frac{1}{s!}\sum\limits_{v_1=0}^{s_1}\cdots\sum\limits_{v_n=0}^{s_n}(-1)^{\sum_{i=1}^{n}v_i} {s_1 \choose v_1}\cdots {s_n \choose v_n}\cdot \lk \sum\limits_{i=1}^n h_i x_i\rk^s\,,
	\]
where  $h_i=s_i/2-v_i$.
\end{Lem}
\begin{proof}[Proof of part (i) of Theorem \ref{Thm_1}]
Our goal is to show
\begin{equation}\label{recursion}
\lim_{n\to \infty}\EE[\mid X_{n }\mid ^2]=\lim_{n\to \infty}\EE\lft[X_{0}^2\prod_{k=0}^{n-1}(Z_{k}(\Delta t))^2\rht]=0\,,
\end{equation}
where $Z_k(\Delta t)$ is given by \eqref{e.2.3} and $X_n$ is given by \eqref{q4}.
{To illustrate the idea we assume  $\kappa=2H$.    The case $\kappa>2H$ can be handled analogously and is in fact simpler.}  We divide our proof into three steps.

\noindent{\bf Step 1}: Bound $\EE\mid X_n\mid ^2$ by a   confluent hypergeometric function.

Applying   Lemma \ref{Polarization} with  $s_1=\cdots=s_n=2$ and $s=2n$ we have
\[
 \prod_{k=0}^{n-1}Z^{2}_{k}(\Delta t)=\frac{1}{(2n)!}\sum\limits_{v_1=0}^{2}\cdots\sum\limits_{v_n=0}^{2}(-1)^{\sum_{i=1}^{n}v_i} {s_1 \choose v_1}\cdots {s_n \choose v_n}\cdot \lk \sum\limits_{i=1}^n h_i Z_{i}(\Delta t)\rk^{2n}\,,
\]
with $h_i=1-v_i$. Note that $Z_{i}(\Delta t)=\alpha_i+\beta_i\cdot V_i^H\overset{d}{\sim} N(\mu_i,\sigma_{i})$ with $\mu_i=\alpha_i$ and $\sigma^2_{i}=\beta_{i}^2(\Delta t)^{2H}$.  Thus we have
\begin{eqnarray}\label{A.EE_Prod2}
	\EE\lk\prod_{k=0}^{n-1}Z^{2}_{k}(\Delta t)\rk &\leq& \frac{2^n}{(2n)!}\sum\limits_{v_1=0}^{2}\cdots\sum\limits_{v_n=0}^{2}  \EE\lk \sum\limits_{i=1}^n (1-v_i)\cdot  Z_{i}(\Delta t)\rk^{2n} \nonumber \\
	&=:& \frac{2^n}{(2n)!}\sum\limits_{v_1=0}^{2}\cdots\sum\limits_{v_n=0}^{2}\EE\lk Q_n\rk^{2n} \,,
\end{eqnarray}
where $Q_n=Q_n(v_1,\cdots,v_n;x_1,\cdots,x_n)=\sum_{i=1}^n (1-v_i)\cdot Z_{i}(\Delta t)$.  It is obvious that  $Q_n$ is still a normal
random variable, with mean $\tilde{\mu}_n$ and variance $\tilde{\sigma}^2_n$ given by
\begin{align*}
	\tilde{\mu}_n:=\tilde{\mu}_n(v_1,\cdots,v_n)=\,&\sum_{i=1}^n (1-v_i)\cdot \mu_i=\sum_{i=1}^n (1-v_i)\cdot \alpha_i \,, 
\end{align*}
and 
\begin{align*} 
	\tilde{\sigma}^2_n:=&\tilde{\sigma}^2_n(v_1,\cdots,v_n)=\,  \EE\lk\lft[\sum_{i=1}^n (1-v_i)\cdot \beta_i\cdot V_i^H \rht]^2 \rk\\
	=&\sum_{i,j=1}^{n} (1-v_i)(1-v_j)\cdot\beta_i\beta_j\cdot \EE[V_i^H V_j^H]\,.
\end{align*}
From the  raw moment formula  (\cite[Eq. (17)]{2012Moments}) it follows
\begin{align}\label{A.EE_Xn2}
	\EE\lk Q_n\rk^{2n}=&\frac{2^{n}}{\sqrt{\pi}}\tilde{\sigma}_n^{2n} \Gamma\lc\frac{2n+1}{2}\rc\cdot \Phi\left(-n,\frac{1}{2},-\frac{\tilde{\mu}_n^2}{2\tilde{\sigma}_n^2}\right)\nonumber \\
	=&\frac{2^{n}}{\sqrt{\pi}}\tilde{\sigma}_n^{2n} \Gamma\lc\frac{2n+1}{2}\rc\cdot \exp\lc-\frac{\tilde{\mu}_n^2}{2\tilde{\sigma}_n^2}\rc \Phi\left(n+\frac{1}{2},\frac{1}{2},\frac{\tilde{\mu}_n^2}{2\tilde{\sigma}_n^2}\right)\,,
\end{align}
 where we   used  Kummer's transformation (see e.g. \eqref{C.Kummer_trans} of the appendix): $\Phi(\alpha,\gamma,z)=\exp(z)\Phi(\gamma-\alpha,\gamma,-z)$. Here,  $\Phi(\alpha,\gamma,z)$ is Kummer's confluent hypergeometric function  (see \eqref{C.Kummer_def} or Chapter 13 in \cite{NIST} for more details).

By employing the differentiation formula   \eqref{C.Kummer_diff}  and then \eqref{C.Kummer_Para2}, we have with the substitution  $\eta=\frac{\tilde{\mu}_n^2}{2\tilde{\sigma}_n^2}$
\begin{align*}
	\frac{d}{d\eta}\bigg[ e^{-\eta}\Phi&\left(\frac{n+1}{2},\frac{1}{2},\eta\right) \bigg] \\
	&=\,n\cdot e^{-\eta}\Phi\left(\frac{n+1}{2},\frac{3}{2},\eta\right)\\
	&=\,n\cdot e^{-\eta}\cdot\frac{2^{\frac{n-3}{2}}\Gamma(\frac n2)e^{\frac \eta2}}{\sqrt{2\eta \pi}}\times[U(n-1/2,-\sqrt{2\eta})-U(n-1/2,\sqrt{2\eta})] \\
	&=\,n\cdot e^{- \frac{\tilde{\mu}_n^2}{4\tilde{\sigma}_n^2}}\cdot \frac{2^{\frac{n-3}{2}}\Gamma(\frac n2)}{\sqrt{\pi\tilde{\mu}_n^2/\tilde{\sigma}_n^2}}\times\lk U\lc n-\frac 12,-\sqrt{\tilde{\mu}_n^2/\tilde{\sigma}_n^2}\rc-U\lc n-\frac 12,\sqrt{\tilde{\mu}_n^2/\tilde{\sigma}_n^2}\rc\rk\,.
\end{align*}
Using the identity  \eqref{C.Para}, we have
\begin{align*}
	&U\lc n-\frac 12,-\sqrt{\tilde{\mu}_n^2/\tilde{\sigma}_n^2}\rc-U\lc n-\frac 12,\sqrt{\tilde{\mu}_n^2/\tilde{\sigma}_n^2}\rc \\
	=\ &\frac{e^{- \frac{\tilde{\mu}_n^2}{4\tilde{\sigma}_n^2}}}{\Gamma(n)}\int_{0}^{\infty} w^{n-1}e^{-w^2/2}\cdot\lk\, e^{w\sqrt{\tilde{\mu}_n^2/\tilde{\sigma}_n^2}}-e^{-w\sqrt{\tilde{\mu}_n^2/\tilde{\sigma}_n^2}} \,\rk dw\geq 0\,.
\end{align*}
This implies  $e^{-\eta}\Phi\left(\frac{n+1}{2},\frac{1}{2},\eta\right)$ is an increasing function with respect  to the variable $\eta(=\frac{\tilde{\mu}_n^2}{2\tilde{\sigma}_n^2})$.
Thus, $\EE\lk Q_n\rk^{2n}$ can be bounded   by the  value at  $\tilde{\mu}_n$ with $\mathsf{\tilde{m}}_n:=\tilde{\mu}_n(0,\cdots,0)=\sum_{i=1}^n \alpha_i$  of this  function,  i.e.,
\begin{align}\label{A.EE_Xn2}
	\EE\lk Q_n\rk^{2n}\leq \frac{2^{n}}{\sqrt{\pi}}\tilde{\sigma}_n^{2n} \Gamma\lc\frac{2n+1}{2}\rc\cdot \exp\lc-\frac{\mathsf{\tilde{m}}_n^2}{2\tilde{\sigma}_n^2}\rc \Phi\left(n+\frac{1}{2},\frac{1}{2},\frac{\mathsf{\tilde{m}}_n^2}{2\tilde{\sigma}_n^2}\right)\,.
\end{align}

\noindent{\bf Step 2}: Analysis of the   confluent hypergeometric function in
\eqref{A.EE_Xn2}.

A key ingredient  of our  proof is  to analyze the asymptotic behavior  as $n\to \infty$ of the right hand of \eqref{A.EE_Xn2} and  this is the objective of this step.

We  claim  that there exists a positive constant $C$ which might change from line to line (we shall not point out the universal constants $C$ unless necessary in this article) 
 such that
\begin{align}\label{e.2.7}
\Phi\left(\frac a2 +\frac 14,\frac 12, \frac{z^2}{2}\right)\leq&
C\cdot 2^{\frac a2-\frac 34}\Gamma\Big(\frac a2+\frac 34 \Big) \times \frac{z^{a-\frac 12}\exp(\frac{z^2}{2})}{\Gamma(\frac 12+a)}
\asymp  \frac{z^{a-\frac 12}\exp(\frac{z^2}{2})}{2^{a/2}\Gamma(\frac a2+\frac 14)} \,.
\end{align}
We shall show the key asymptotic  behaviors of the confluent hypergeometric functions $\Phi(a,b,z)$. The idea is motivated by the Poincar\'e-type asymptotic forms \eqref{C.Kummer_Asymp} of confluent hypergeometric function. In our case, we have  $a=2n+\frac 12$, so the parameter $a$ can also goes to infinity. Fortunately,  we have $z^{2}=\frac{\mathsf{\tilde{m}}_n^2}{\tilde{\sigma}_n^2}\geq C\cdot n^{2H}$ (see the proof in the Appendix A) , the parameter $a$ is majored by $z$ since $H>1/2$. 

To prove the claim \eqref{e.2.7} we  employ the integral representation of the parabolic cylinder functions \eqref{C.Para}.
  For $z>0$ the parabolic cylinder functions are computed as follows:
\begin{align}
	U(a,z)&=\frac{z^{a+\frac 12}\exp(-\frac{z^2}{4})}{\Gamma(\frac 12+a)}\int_{0}^{\infty} t^{a-\frac 12} \exp(-z^2(\frac 12 t^2+t))dt\nonumber \\
	&=\frac{z^{a+\frac 12}\exp(\frac{z^2}{4})}{\Gamma(\frac 12+a)}\int_{1}^{\infty} (s-1)^{a-\frac 12}\exp(-\frac{z^2 s^2}{2})ds
	\label{e.2.8}
\end{align}
and
\begin{align}
	U(a,-z)&=\frac{z^{a+\frac 12}\exp(-\frac{z^2}{4})}{\Gamma(\frac 12+a)}\int_{0}^{\infty} t^{a-\frac 12} \exp(-z^2(\frac 12 t^2-t))dt\nonumber  \\
	&=\frac{z^{a+\frac 12}\exp(\frac{z^2}{4})}{\Gamma(\frac 12+a)}\int_{-1}^{\infty} (s+1)^{a-\frac 12}\exp(-\frac{z^2 s^2}{2})ds\,.\label{e.2.9}
\end{align}
The sum of the integrals in \eqref{e.2.8}-\eqref{e.2.9}  can be    dominated as
follow
(with  $a=2n+\frac 12$ and $z^{2}=\frac{\mathsf{\tilde{m}}_n^2}{\tilde{\sigma}_n^2}$).
\begin{align*}
	\int_{1}^{\infty} (s-1)^{a-\frac 12}&\exp(-\frac{z^2 s^2}{2})ds+\int_{-1}^{\infty} (s+1)^{a-\frac 12}\exp(-\frac{z^2 s^2}{2})ds \\
	\leq\, & 2\int_{-1}^{\infty} (s+1)^{a-\frac 12}\exp(-\frac{z^2 s^2}{2})ds=2\int_{-1}^{\infty} (s+1)^{2n}\exp(-\frac{\mathsf{\tilde{m}}_n^2 s^2}{2\tilde{\sigma}_n^2})ds \,.
\end{align*}
Basically, we know that $z^{2}=\frac{\mathsf{\tilde{m}}_n^2}{\tilde{\sigma}_n^2}\geq C\cdot  n^{2H}\ge  n$ for $n$ large enough. So, for sufficient large $n$
\[
 (s+1)^{2(n+1)}\leq \exp \left(2(n+1)s \right)\leq \exp \left(\frac{\mathsf{\tilde{m}}_n^2 s^2}{4\tilde{\sigma}_n^2} \right)
\]
for all $s\geq -1$. Therefore, we can easily obtain
\begin{align*}
	\int_{-1}^{\infty} (s+1)^{2n}\exp\left(-\frac{\mathsf{\tilde{m}}_n^2 s^2}{2\tilde{\sigma}_n^2}\right)ds \leq& \int_{-1}^{\infty} \exp\left(-\frac{\mathsf{\tilde{m}}_n^2 s^2}{4\tilde{\sigma}_n^2}\right)ds \\
	\leq& \int_{\RR} \exp\left(-\frac{\mathsf{\tilde{m}}_n^2 s^2}{4\tilde{\sigma}_n^2}\right)ds=C\cdot \frac{\tilde{\sigma}_n}{\mathsf{\tilde{m}}_n}=C\cdot \frac{1}{z}\,.
\end{align*}

Recall the relation between $\Phi(a,b,z)$ and the parabolic cylinder functions $U(a,z)$ given by \eqref{C.Kummer_Para1}. As a result, we get
\begin{align}
	\Phi\Big(\frac a2 +\frac 14,\frac 12, \frac{z^2}{2}\Big)&=\frac{2^{\frac a2-\frac 34}}{\sqrt{\pi}}\Gamma\Big(\frac a2+\frac 34 \Big)\exp(\frac{z^2}{4})\times[U(a,z)+U(a,-z)] \nonumber\\
	&\leq C\cdot 2^{\frac a2-\frac 34}\Gamma\Big(\frac a2+\frac 34 \Big)\exp(\frac{z^2}{4})\times \frac{z^{a+\frac 12}\exp(\frac{z^2}{4})}{\Gamma(\frac 12+a)}\times \frac{1}{z} \nonumber \\
	&= C\cdot 2^{\frac a2-\frac 34}\Gamma\Big(\frac a2+\frac 34 \Big) \times \frac{z^{a-\frac 12}\exp(\frac{z^2}{2})}{\Gamma(\frac 12+a)}\,. \label{e.2.4}
\end{align}
Thus we finish the proof of our  claim  \eqref{e.2.7}.

\noindent{\bf Step 3}: Completion of the proof of part (i) of Theorem \ref{Thm_1}.


Applying \eqref{e.2.4}
with $a=2n+\frac 12$, $z^{2}=\frac{\mathsf{\tilde{m}}_n^2}{2\tilde{\sigma}_n^2}\geq C\cdot \frac{n^{2}}{n^{2-2H}}=   n^{2H}\ge n$ we obtain 
\begin{align}\label{A.EE_Prod.b2}
	\EE\lk\prod_{k=1}^{n}Z^{2}_{k}(\Delta t)\rk \leq&\, \frac{2^n}{(2n)!}\sum\limits_{v_1=0}^{2}\cdots\sum\limits_{v_n=0}^{2}\EE\lk Q_n\rk^{2n} \nonumber\\
\leq&\, \frac{2^n\cdot 3^n}{(2n)!}\cdot \frac{2^{n}}{\sqrt{\pi}}\tilde{\sigma}_n^{2n} \Gamma\lc\frac{2n+1}{2}\rc\cdot \exp\lc-\frac{\mathsf{\tilde{\mu}}_n^2}{2\tilde{\sigma}_n^2}\rc
\Phi\left(n+\frac{1}{2},\frac{1}{2},\frac{\mathsf{\tilde{\mu}}_n^2}{2\tilde{\sigma}_n^2}\right) \nonumber\\
	\leq&\, \frac{2^n\cdot 3^n}{(2n)!}\cdot \frac{2^{n}}{\sqrt{\pi}}\tilde{\sigma}_n^{2n} \Gamma\lc\frac{2n+1}{2}\rc\cdot \exp\lc-\frac{\mathsf{\tilde{m}}_n^2}{2\tilde{\sigma}_n^2}\rc \Phi\left(n+\frac{1}{2},\frac{1}{2},\frac{\mathsf{\tilde{m}}_n^2}{2\tilde{\sigma}_n^2}\right) \nonumber\\
	\leq&\, \frac{2^n\cdot 3^n}{(2n)!}\cdot \frac{2^{n}}{\sqrt{\pi}}\tilde{\sigma}_n^{2n} \Gamma\lc\frac{2n+1}{2}\rc\cdot \exp\lc-\frac{\mathsf{\tilde{m}}_n^2}{2\tilde{\sigma}_n^2}\rc \nonumber \\
	&\qquad\quad \cdot \frac{C}{\Gamma(n+\frac{1}{2})}\lc\frac{\mathsf{\tilde{m}}_n^2}{2\tilde{\sigma}_n^2}\rc^{n} \exp\lc\frac{\mathsf{\tilde{m}}_n^2}{2\tilde{\sigma}_n^2}\rc \nonumber\\
	\asymp&\, \frac{6^n\cdot\mathsf{\tilde{m}}_n^{2n}}{(2n)!}\asymp \frac{6^n\cdot n^{2n}}{\sqrt{4\pi n}\cdot(2n/e)^{2n}}\lc\frac{1-\theta}{\theta}\rc^{2n} \asymp \frac{(\sqrt{3/2}e)^{2n}}{\sqrt{4\pi n}}\lc\frac{1-\theta}{\theta}\rc^{2n}\,,
\end{align}
by Stirling's approximation, where we   apply the  claim \eqref{e.2.7}  in the above forth
inequality
and  the fact that $\frac{1}{2}<\theta<1$ in the above  last inequality.  Now, it is obvious  to see from \eqref{A.EE_Prod.b2} that $\EE\lk\prod_{k=1}^{n}Z^{2}_{k}(\Delta t)\rk\to 0$ as $n\to \infty$ if
\[
 \sqrt{3/2}e\cdot \frac{1-\theta}{\theta} \leq 1 \,\Leftrightarrow\, \theta \geq \frac{\sqrt{3/2}\cdot e}{\sqrt{3/2}\cdot e+1}\approx 0.77\,,
\]
proving   part (i) of Theorem \ref{Thm_1}.
\end{proof}

\begin{Rem}
 We believe our method can also work under the  condition that $X_0= 0$ with probability $0$ and  $\EE[\mid X_0\mid ^2]<\infty$.    For example, one can apply H\"older inequality to \eqref{recursion} and then follow the same argument  there. But this makes the computations much more involved.
 We are not pursuing the detail along this direction to simplify our presentation.
\end{Rem}

\begin{Rem}\label{Rem.2.2}
Following the same strategy as in our proof, we can prove more general results:  For any integer $p\geq 2$,
 if $\frac{1}{(1+M_{p})}\leq \theta <1$,
 where $M_{p}=\frac{2}{e}\cdot\frac{1}{(p+1) {p\choose p/2}}
 $,  then $\lim\limits_{n\to \infty}\mathbb{E}(X_{n}^{p})\to 0$.
\end{Rem}

\subsection{The case of $\kappa> 3/2$ and $\frac{1}{2}< \theta\leq 1$}\label{subsection 2.3}
In this subsection we shall prove part (ii) of Theorem \ref{Thm_1}. To begin with, let us recall the celebrated Gaussian correlation inequality.
\begin{Lem}{\cite[Theorem 2]{LM2017}}\label{GCI}
	Let $n=n_1+n_2$ and $X$ be an n-dimensional centered Gaussian vector. Then for any $t_1,\cdots,t_n>0$,
	\begin{align*}
		\mathbbm{P}\{&\mid X_1\mid\leq t_1,\dots,\mid X_n\mid\leq t_n\} \\
		&\geq \mathbbm{P}\{\mid X_1\mid\leq t_1,\dots,\mid X_{n_1}\mid\leq t_{n_1}\}\cdot \mathbbm{P}\{\mid X_{n_1+1}\mid\leq t_{n_1+1},\dots,\mid X_{n}\mid\leq t_{n}\}\,.
	\end{align*}
\end{Lem}

\begin{proof}[Proof of part (ii) of Theorem \ref{Thm_1}]
Let us consider
\begin{align}
	X_{n+1}^{2}=&\prod_{k=1}^{n}(Z_{k}(\Delta t))^{2}=\prod_{k=1}^{p(n)}(Z_{k}(\Delta t))^{2}\cdot \prod_{k=p(n)+1}^{n}(Z_{k}(\Delta t))^{2} \nonumber\\
	 &{ =\prod_{k=1}^{p(n)}(Z_{k}(\Delta t))^{2}\cdot \prod_{k=p(n)+1}^{n}(Z_{k}(\Delta t))^{2} \cdot \mathbbm{1}_{\{Z_{k}(\Delta t) \leq 0:~p(n)+1\leq k\leq n \}}
  	} \label{A.Rec1} \\
	&\qquad\qquad\qquad { +  \left[ \prod_{k=1}^{n}(Z_{k}(\Delta t))^{2}
	\right] \cdot  \left[ 1-\prod_{k=p(n)+1}^{n} \  \mathbbm{1}_{\{Z_{k}(\Delta t) \leq 0:~p(n)+1\leq k\leq n \}} \right]\,,  } \label{A.Rec2}
\end{align}
with $p(n)=n/p$, $q(n)=n/q$ and $1/p+1/q=1$. Formally, we know $Z_{k}(\Delta t)$ converges to $c_\theta=-\frac{1-\theta}{\theta}<0$. Thus, the probability of the event $\{Z_{k}(\Delta t) \leq 0:~p(n)+1\leq k\leq n \}$ converges to one.	

Firstly, applying the following inequality of  bounding  the    geometric mean by the arithmetic one
\begin{empheq}[left=\empheqlbrace]{align}
&a_0^{p_0}a_1^{p_1}\cdots a_n^{p_n}\leq \left(\frac{p_0a_0+p_1a_1+\cdots+p_na_n}{p_0+p_1+\cdots +p_n}\right)^{p_0+p_1+\cdots+p_n}\nonumber\\
&a_0,\cdots,a_n \geq 0,~ p_0,p_1,\cdots,p_n\in \mathbb{N^{+}},
\quad \text{with~} p_0=\cdots=p_n=2\nonumber
\end{empheq}
 to the second factor of \eqref{A.Rec1} yields
\begin{align}\label{A.Ineq.AG}
	\bar X_{n+1}^{2}:=&
 	\prod_{k=1}^{p(n)}(Z_{k}(\Delta t))^{2}\cdot
	\prod_{k=p(n)+1}^{n}(-Z_{k}(\Delta t))^{2} \cdot \mathbbm{1}_{\{Z_{k}(\Delta t) \leq 0:~p(n)+1\leq k\leq n \}} \nonumber\\
	\leq&
 	\prod_{k=1}^{p(n)}(Z_{k}(\Delta t))^{2}\cdot
	\left[\frac{1}{n-p(n)}\sum_{k=p(n)+1}^{n}Z_{k}(\Delta t)\cdot \mathbbm{1}_{\{Z_{k}(\Delta t) \leq 0:~p(n)+1\leq k\leq n \}} \right]^{2n}.
\end{align}
By the H\"older inequality, we then have
\begin{align}
\EE \bar X_{n+1}^{2}	\leq& \lc\EE\prod_{k=1}^{p(n)}(Z_{k}(\Delta t))^{4}\rc^{\frac 12}\cdot \lc\EE\left[\frac{1}{n-p(n)}\sum_{k=p(n)+1}^{n}Z_{k}(\Delta t) \right]^{8n}\rc^{\frac 14} \nonumber\\
	&\qquad\qquad\qquad\qquad~ \cdot\Big( \mathbbm{P}{\{Z_{k}(\Delta t) \leq 0:~p(n)+1\leq k\leq n \}}\Big)^{\frac 14} \nonumber\\
	\leq& \lc\EE\prod_{k=1}^{p(n)}(Z_{k}(\Delta t))^{4}\rc^{\frac 12}\cdot \lc\EE\left[\frac{1}{n-p(n)}\sum_{k=p(n)+1}^{n}Z_{k}(\Delta t) \right]^{8n}\rc^{\frac 14}\,.\label{eq.Rec1}
\end{align}
By the same methods as  in the proof of part (i),   there is an $M>1$ such that the first factor of \eqref{eq.Rec1}
can be bounded by
\begin{equation}
 \lc\EE\prod_{k=1}^{p(n)}(Z_{k}(\Delta t))^{4}\rc^{\frac 12}\leq M^{p(n)}=(M^{1/p})^n\,.
\end{equation}
For the second term in  \eqref{eq.Rec1}, we have from Kummer's transformation \eqref{C.Kummer_trans}
\begin{align}\label{eq.Rec1_2}
\EE\bigg( \bigg[\frac{1}{n-p(n)}&\sum_{k=p(n)+1}^{n}Z_{k}(\Delta t)\bigg]^{8n} \bigg) \nonumber\\
=\,& \frac{C}{\sqrt{\pi}}\bar{\sigma}_{p(n)}^{8n}2^{4n} \Gamma\lc\frac{8n+1}{2}\rc\cdot
\Phi\left(-4n,\frac{1}{2};-\frac{\bar{\mu}_{p(n)}^2}{2\bar{\sigma}_{p(n)}^2}\right) \nonumber\\
=\,& \frac{C}{\sqrt{\pi}}\bar{\sigma}_{p(n)}^{8n}2^{4n} \Gamma\lc\frac{8n+1}{2}\rc\cdot \exp\lc-\frac{\bar{\mu}_{p(n)}^2}{2\bar{\sigma}_{p(n)}^2}\rc
\Phi\left(4n+\frac{1}{2},\frac{1}{2};-\frac{\bar{\mu}_{p(n)}^2}{2\bar{\sigma}_{p(n)}^2}\right)\,,
\end{align}
where $\bar{\mu}_{p(n)}:=\frac{1}{n-p(n)}\sum_{k=p(n)+1}^n \alpha_k$ and $\bar{\sigma}_{p(n)}^2:=\EE\lk \lvert \frac{1}{n-p(n)}\sum_{k=p(n)+1}^n \beta_k\cdot V_k^H  \rvert^2 \rk$. Then by the same procedure as in Step 2, we can bound \eqref{eq.Rec1_2}, namely, the second factor of \eqref{eq.Rec1}  by the following:
\begin{align}\label{eq.Rec1_22}
&
C \bar{\sigma}_{p(n)}^{8n}2^{4n} \Gamma\lc\frac{8n+1}{2}\rc \exp\lc-\frac{\bar{\mu}_{p(n)}^2}{2\bar{\sigma}_{p(n)}^2}\rc\cdot\frac{\Gamma(1/2)}{\Gamma(4n+1/2)}\lc\frac{\bar{\mu}_{p(n)}^2}{2\bar{\sigma}_{p(n)}^2}\rc^{4n}\exp\lc\frac{\bar{\mu}_{p(n)}^2}{2\bar{\sigma}_{p(n)}^2}\rc \nonumber\\
&\qquad \leq\, C\bar{\mu}^{q(n)}{ \leq C \left( \frac{1-\theta}{\theta} \right)^{q(n)}}=C \left( \frac{1-\theta}{\theta} \right)^{n/q} \to 0\,, \qquad\qquad\qquad(\text{as~} n\to \infty)\,.
\end{align}
Combining \eqref{eq.Rec1}-\eqref{eq.Rec1_22} we conclude  for  the first summand  \eqref{A.Rec1}
\begin{align*}
\EE \left[\bar X_n^2\right] 	&\leq M^{p(n)}\cdot \mid c_{\theta}\mid ^{q(n)}\to 0\,,& \\
	&\Leftrightarrow M^{1/p}\cdot \mid c_{\theta}\mid ^{1/q}<1 \,\Leftrightarrow\, p>\frac{\ln(\mid c_{\theta}\mid )-\ln(M)}{\ln(\mid c_{\theta}\mid )}>1\,.&
\end{align*}
Next, we  treat   \eqref{A.Rec2}.  It is easy to see that  if $\beta_n>0$ (i.e. $\lambda<0,~ \mu>0$), then
\begin{align*}
	\mathbbm{P}\{Z_{k}(\Delta t) \leq 0\}=&\mathbbm{P}\{V^H_{k} \leq -\alpha_{k}/\beta_{k}=-C_{\lambda,\mu,\Delta t}(n)\}=\mathbbm{P}\{V^H_{k} \geq C_{\lambda,\mu,\Delta t}(n)\}\,,
\end{align*}
and if  $\beta_n<0$ (i.e. $\lambda<0,~ \mu<0$), then
\begin{align*}
	\mathbbm{P}\{Z_{k}(\Delta t) \leq 0\}=&\mathbbm{P}\{V^H_{k} \geq -\alpha_{k}/\beta_{k}=-C_{\lambda,\mu,\Delta t}(n)\}=\mathbbm{P}\{V^H_{k} \leq C_{\lambda,\mu,\Delta t}(n)\}\,,
\end{align*}
where $C_{\lambda,\mu,\Delta t}(n):=\frac{1+\kappa(1-\theta)\lambda n^{\kappa-1} \Delta t^{\kappa}}{\mu}$. Consequently, we have by the classical concentration inequality for normal variable $V^H_{k}$
\begin{align}\label{A.ConcIneq}
	\mathbbm{P}\{Z_{k}(\Delta t) \leq 0\}\geq&\mathbbm{P}\{\mid V^H_{k}\mid  \leq \mid C_{\lambda,\mu,\Delta t}(n)\mid \} \geq1-2\exp\lk-\frac{\mid C_{\lambda,\mu,\Delta t}(n)\mid ^2}{2(\Delta t)^{2H}}\rk \,.
\end{align}
Then, by the 
Gaussian correlation inequality (Lemma \ref{GCI}), we can get
\begin{align}\label{A.GCIneq}
	\mathbbm{P}\{Z_{k}(\Delta t) \leq 0:~p(n)+1\leq k\leq n \}\geq& \mathbbm{P}\{\mid V^H_{k}\mid  \leq \mid C_{\lambda,\mu,\Delta t}(n)\mid :~p(n)+1\leq k\leq n\} \nonumber\\
	\geq& \prod_{k=p(n)+1}^{n} \mathbbm{P}\{\mid V^H_{k}\mid  \leq \mid C_{\lambda,\mu,\Delta t}(n)\mid \} \nonumber \\
	\geq& \prod_{k=p(n)+1}^{n}\lc 1-2\exp\lk-\frac{\mid C_{\lambda,\mu,\Delta t}(n)\mid ^2}{2(\Delta t)^{2H}}\rk \rc\,.
\end{align}
Denote
\begin{align}
%
& \tilde X_n :={  1- \prod_{k=p(n)+1}^{n} \  \mathbbm{1}_{\{Z_{k}(\Delta t) \leq 0:~p(n)+1\leq k\leq n \}} } \,. \nonumber
\end{align}
By the Weierstrass  product inequality:
\begin{align*}
	\prod_{i=1}^{n}(1-x_i)>1-\sum_{i=1}^{n}x_i\,,\quad \forall \quad \hbox{$ x_1, \cdots, x_n \in (0, 1)$}\,,
\end{align*}
we have
\begin{align}\label{A.GCIneqW}
\EE \left[\tilde X_n\right]	&\ \leq   1-\prod_{k=p(n)+1}^{n}\lc 1-2\exp\lk-\frac{\mid C_{\lambda,\mu,\Delta t}(n)\mid ^2}{2(\Delta t)^{2H}}\rk \rc\nonumber \\
	&\ \leq  2\sum_{k=p(n)+1}^{n}\exp\lk-\frac{\mid C_{\lambda,\mu,\Delta t}(n)\mid ^2}{2(\Delta t)^{2H}}\rk \lesssim \exp\lk-\frac{\mid C_{\lambda,\mu,\Delta t}(p(n))\mid ^2}{2(\Delta t)^{2H}}\rk\,,
\end{align}
since when  $n$ is sufficiently  large   that $2\exp[-\mid C_{\lambda,\mu,\Delta t}(n)\mid ^2/2(\Delta t)^{2H}]<1$.
Because $\tilde X_n $ is either $0$  or $1$,
i.e. $\tilde X_n^2= \tilde X_n $ we have   for the second summand  \eqref{A.Rec2}
\begin{align*}
	&\EE\left\{\left[\prod_{k=1}^{n}(Z_{k}(\Delta t))^{2}\right]   \cdot   \tilde X_n
	\right\}
	 \leq \lc\EE\bigg[\prod_{k=1}^{n}(Z_{k}(\Delta t))^{4}\bigg]\rc^{\frac 12} \cdot \Big( \EE (\tilde X_n) \Big)^{\frac 12} \\
	&\qqquad \leq C M^{2n}\cdot \exp\lk-\frac{\mid C_{\lambda,\mu,\Delta t}(p(n))\mid ^2}{2(\Delta t)^{2H}}\rk \asymp M^{2n} \exp\lk-\frac{\lambda^2}{\mu^2}\cdot p(n)^{2(\kappa-1)} (\Delta t)^{2\kappa-2H}\rk\,.
\end{align*}
Here we   applied $\EE\bigg[\prod_{k=1}^{n}(Z_{k}(\Delta t))^{4}\bigg]\leq M^{4n}$ for some constant $M>1$, which can be proved  analogously  as in  the proof of part (i) of the theorem. Hence,   it is easy to see if $\kappa>3/2$, $p(n)^{2(\kappa-1)}\geq C_p\cdot n^{2(\kappa-1)}\gg n$, then the above term converges to $0$.
\end{proof}

\subsection{The case of $0<\theta<\frac 1 2$}\label{subsection 2.4}
\label{section 2.2}
In this subsection we   prove part (iii) of Theorem \ref{Thm_1}.  First, we state the following
strong law of large numbers (SLLN).

\begin{Lem}{\cite[Theorem 1]{HRV2008}}\label{SLLN}
Let $\xi_1, \xi_2, \cdots, \xi_n$ be  a sequence of square-integrable random variables and suppose that there exists a sequence of constants $\sR_k$ such that
\begin{equation}\label{Cond.Upsilon}
	\sup_{n\geq 1}\mid \mathrm{Cov}(\xi_n,\xi_{n+k})\mid \leq \sR_k\,,~k\geq 1\,, \qquad \sum_{k=1}^{\infty} \frac{\sR_k}{k^q}<\infty\quad \text{for some}~0\leq q<1\,,
\end{equation}
and
\begin{equation}\label{Cond.Var}
	\sum_{k=1}^{\infty}\frac{\mathbb{V}(\xi_n)\cdot [\log(n)]^2}{n^2}<\infty\,,
\end{equation}
then the SLLN holds. More precisely, letting $S_n=\sum_{i=0}^{n}\xi_i$, one has
\begin{equation}\label{eq.SLLN}
	\lim_{n\to \infty} \frac{S_n-\mathbb{E}(S_n)}{n}=0 \quad\text{almost surely}\,.
\end{equation}
\end{Lem}

With the help of this lemma we now  give 
the  proof of the last part of the theorem.

\begin{proof}[Proof of part (iii) of Theorem \ref{Thm_1}]
Denote  $Y_0=\ln X_0^2$, $Y_{k}=\ln\left(\alpha_k+\beta_kV_k^H\right)^2$ and $S_n=\sum\limits_{k=0}^{n}Y_k$. In the above definition  if $\alpha_k+\beta_k V_k^H=0$, then we put $Y_k:=0$.  Notice  that $(\alpha_k  +\beta_k V_k^H )^2$ is positive almost surely, so $Y_k$ are well defined for $k\geq 0$.
  We shall apply  Lemma \ref{SLLN} to  $\xi_n=Y_{n}$.   It is easier  to verify that \eqref{Cond.Var} holds. The main objective is to verify the conditions in \eqref{Cond.Upsilon}. For $q\in(2H-1,1)$,
the second condition of \eqref{Cond.Upsilon} holds if $\sR_k\asymp \mid k\mid ^{2H-2}$ for   sufficiently large $k$.
 Thus, the proof of part (iii) in Theorem \ref{Thm_1}
  is completed if we can show for some constant $C$
%
\begin{equation}\label{Ineq.Cov}
\sup_{n\ge 1}	\mid \mathrm{Cov}(Y_{n},Y_{n+k})\mid \leq \sR_k\leq  C\cdot \mid k\mid ^{2H-2}\,.
\end{equation}
In fact, assume \eqref{Ineq.Cov} and    recall that if $0<\theta<\frac 1 2$, then
$$
\lim_{n\to \infty}\frac{1}{n}\mathbb{E}(S_n)=\lim_{n\to \infty}\frac{1}{n}\sum_{k=1}^{n}
\mathbb{E}\left[\ln(\alpha_k+\beta_kV_{k}^{H})^{2}\right]= C\cdot\ln\left(\frac{1-\theta}{\theta}\right)^2=:C_\theta>0\,.
$$
 Therefore, by Lemma \ref{SLLN} with $q\in (2H-1, 1)$, we get
$$
\frac{S_n}{n}=\frac{\ln(X_{n})^2}{n}\overset{a.s.}\longrightarrow C_\theta >0\,.
$$
This  implies $(X_{n})^2\to \infty$ almost surely. Consequently, by Fatou's Lemma, one has
$$
\varliminf_{n\to \infty} \mathbb{E}\mid X_{n}\mid ^2\geq \mathbb{E}\left[\varliminf_{n\to \infty}\mid X_{n}\mid ^2\right]=\infty\,,
$$
which completes the proof of part (iii) of the theorem.

So, it suffices to show \eqref{Ineq.Cov}.  We shall show
that $\mid \mathrm{Cov}(Y_i,Y_j)\mid \lesssim \mid i-j\mid ^{2H-2}$ as $\mid i-j\mid \to \infty$ which is obviously  equivalent to \eqref{Ineq.Cov}.
In fact,
\[
\mathrm{Cov}(Y_i,Y_j)=\mathbb{E}\left[\ln(\alpha_i+\beta_iV_i^H)^2\ln(\alpha_j+\beta_jV_j^H)^2\right]
-\mathbb{E}\left[\ln(\alpha_i+\beta_iV_i^H)^2\right]\mathbb{E}\left[\ln(\alpha_j+\beta_jV_j^H)^2\right]\,.
\]
Denote the probability densities of normal variables $V_i^H$ and $V_j^H$ by $f_i(x)$ and $f_j(y)$, and denote the corresponding cumulative distributions by $F_i(x)$ and $F_j(y)$ respectively.  The symmetric covariance matrix of $V_i^H$ and $V_j^H$  is given by
\[
\Sigma
=\left(
         \begin{array}{cc}
           \sigma_i^2, & \rho_{ij} \sigma_i \sigma_j \\
           \rho_{ij} \sigma_i \sigma_j, & \sigma_j^2 \\
         \end{array}
       \right)\,,
\]
where $\sigma_i:=\sqrt{\mathbb{E}[(V_{i}^{H})^2]}=\mid t_{i+1}-t_{i}\mid ^H=\mid \Delta t\mid ^H$ and $\sigma_j:=\sqrt{\mathbb{E}[(V_{j}^{H})^2]}=\mid t_{j+1}-t_{j}\mid ^H=\mid \Delta t\mid ^H$ are standard deviations of $V_i^H$ and $V_j^H$, $\rho_{ij}:=\frac{\mathbb{E}[V_i^HV_j^H]}{\sigma_i \sigma_j}$  is the correlation coefficient
between $V_i^H$ and $V_j^H$. Their  joint distribution has the following   form
\begin{align}\label{eq.density}
	f_{i,j}(x,y)=\,&\frac{1}{\sqrt{(2\pi)^2\det(\Sigma)}}\exp\left(-\frac{X^{T}\Sigma^{-1}X}{2}\right) \nonumber \\
	=\,&\frac{1}{2\pi \sigma_i \sigma_j \sqrt{1-\rho_{ij}^2}}\exp\lc -\frac{1}{2(1-\rho_{ij}^2)}\lk \frac{x^2}{\sigma_i^2}-2\rho_{ij} \frac{x\cdot y}{\sigma_i\cdot  \sigma_j}+\frac{y^2}{\sigma_j^2} \rk  \rc\,,
\end{align}
with $X=[x,y]^{T}$.
Without loss of generality, we can assume that $i\geq j+1$. Then we have  using  the  joint density \eqref{eq.density}:
\begin{align*}
	\mathrm{Cov}(Y_i,Y_j)=& \int_{\mathbb{R}^2} \left[\ln(\alpha_i+\beta_i x)^2 \ln(\alpha_j+\beta_j y)^2 \right]\cdot\left[f_{i,j}(x,y)-f_i(x)f_j(y)\right] dxdy\\
	=&\int_{\mathbb{R}^2} \left[\ln(\alpha_i+\beta_i x)^2 \ln(\alpha_j+\beta_j y)^2 \right]\cdot \exp\lc-\frac{\tilde{\rho}_{ij}}{2}\lk \frac{x^2}{\sigma_i^2}+\frac{y^2}{\sigma_j^2} \rk \rc \\
	&\times \lk -\exp\lc\frac{\tilde{\rho}_{ij}}{2}\lk \frac{x^2}{\sigma_i^2}+\frac{y^2}{\sigma_j^2} \rk \rc+\frac{1}{\sqrt{1-\rho_{ij}^2}}\exp\lc {\bar{\rho}_{ij}}\cdot \frac{xy}{\sigma_i \sigma_j} \rc \rk  dF_i(x)dF_j(y)\,,
\end{align*}
where $\tilde{\rho}_{ij}=\frac{\rho_{ij}^2}{1-\rho_{ij}^2}$, $\bar{\rho}_{ij}=\frac{\rho_{ij}}{1-\rho_{ij}^2}$.
By  the  H\"{o}lder inequality, $\mathrm{Cov}(Y_i,Y_j)$ can be bounded by
\begin{align*}
	& \lc \int_{\mathbb{R}^2} \left[\ln(\alpha_i+\beta_i x)^2 \ln(\alpha_j+\beta_j y)^2 \right]^2 dF_i(x)dF_j(y) \rc^{\frac 12} \\
	&\times\bigg( \int_{\mathbb{R}^2} \bigg[ \exp\lc\frac{\tilde{\rho}_{ij}}{2}\lk \frac{x^2}{\sigma_i^2}+\frac{y^2}{\sigma_j^2} \rk \rc \\
	&\qquad\qquad\qquad-\frac{1}{\sqrt{1-\rho_{ij}^2}}\exp\lc{\bar{\rho}_{ij}}\cdot \frac{xy}{\sigma_i \sigma_j} \rc \bigg]^2 dF_i(x)dF_j(y) \bigg)^{\frac{1}{2}}=:A_{ij}^{\frac 12}\times B_{ij}^{\frac 12}\,.
\end{align*}
We proceed to estimate $A_{ij}$ and $B_{ij}$. To estimate $A_{ij}$ we only need to consider
\begin{align*}
	&\int_{\RR} \lk \ln(\alpha+\beta x)^2\rk^2 \times \frac{1}{\sqrt{2\pi \sigma}}\exp\lc-\frac{x^2}{2\sigma^2} \rc dx
	\asymp \int_{\RR} \lk \ln(\alpha+\beta\sigma x)^2\rk^2 \times e^{-\frac{x^2}{2}}dx \\
	=&\int_{\mid \alpha+\beta\sigma x\mid \leq 1}\lk \ln(\alpha+\beta\sigma x)^2\rk^2 \times e^{-\frac{x^2}{2}}dx+\int_{\mid \alpha+\beta\sigma x\mid \geq 1} \lk \ln(\alpha+\beta\sigma x)^2\rk^2 \times e^{-\frac{x^2}{2}}dx \\
	\leq&\int_{\mid \alpha+\beta\sigma x\mid \leq 1}\lk \ln(\alpha+\beta\sigma x)^2\rk^2 dx + \int_{\mid \alpha+\beta\sigma x\mid \geq 1} \lk \alpha+\beta\sigma x \rk^4 \times e^{-\frac{x^2}{2}}dx\,.
\end{align*}
Here, we neglect the subscripts of $\alpha_i$ and $\beta_i$ to simplify the notations. Obviously, there exists a constant $C$ such that $A_{i,j}\le  C$  since $\sigma=\mid \Delta  t\mid ^{ H} $ and $\alpha$ and $\beta$ defined by \eqref{e.2.1}-\eqref{e.2.2} are bounded above and bounded below away from $0$. Next, for $\mid i-j\mid \to \infty$, we deal with $B_{ij}$:
\begin{align*}
	&B_{ij}=\int_{\mathbb{R}^2} \Bigg[ \exp\lc\frac{\tilde{\rho}_{ij}}{2}\lk \frac{x^2}{\sigma_i^2}+\frac{y^2}{\sigma_j^2} \rk \rc \\
	&\qquad\qquad\quad-\frac{1}{\sqrt{1-\rho_{ij}^2}}\exp\lc \frac{\bar{\rho}_{ij}\cdot xy}{\sigma_i \sigma_j} \rc \Bigg]^2 \times \frac{1}{2\pi \sigma_i \sigma_j }\exp\lc -\frac{1}{2}\lk \frac{x^2}{\sigma_i^2}+\frac{y^2}{\sigma_j^2} \rk  \rc dxdy.
\end{align*}
By variable substitutions $x\rightarrow \sqrt 2\sigma_i x, y\rightarrow \sqrt 2 \sigma_j y$,  we have
\begin{align*}
	B_{ij}\leq\ &C\int_{\mathbb{R}^2} \bigg[\exp\lc\tilde{\rho}_{ij}\lk x^2+y^2 \rk \rc -\frac{1}{\sqrt{1-\rho_{ij}^2}} \exp\lc {2{\bar{\rho}_{ij}}\cdot xy} \rc \bigg]^2 \times \exp\lc -\lk x^2+y^2 \rk  \rc dxdy\\
	=\ &C\int_{\mathbb{R}^2} \bigg[\exp\lc 2\tilde{\rho}_{ij}\lk x^2+y^2 \rk \rc+\frac{1}{{1-\rho_{ij}^2}}\exp\lc {4{\bar{\rho}_{ij} }\cdot xy} \rc \\
	&\qquad\qquad\quad -\frac{2}{\sqrt{1-\rho_{ij}^2}}\exp\lc \tilde{\rho}_{ij}\lk x^2+y^2\rk+2\bar{\rho}_{ij}xy  \rc \bigg] \times \exp\lc -\lk x^2+y^2 \rk  \rc dxdy\,.
\end{align*}
The above three integrals can  be  explicitly   evaluated as
follows:
\[
 \int_{\mathbb{R}^2} \exp\lc 2\tilde{\rho}_{ij}\lk x^2+y^2 \rk \rc \times \exp\lc -\lk x^2+y^2 \rk  \rc dxdy=\frac{\pi}{1-2\tilde{\rho}_{ij}} \,,
\]
\[
 \frac{1}{{1-\rho_{ij}^2}}\int_{\mathbb{R}^2}\exp\lc {4\bar{\rho}_{ij} \cdot xy} \rc \times \exp\lc -\lk x^2+y^2 \rk  \rc dxdy=\frac{1}{{1-\rho_{ij}^2}}\cdot\frac{\pi}{\sqrt{1-4{\bar{\rho}_{ij}}^2}} \,,
\]
and
\begin{align*}
	\frac{2}{\sqrt{1-\rho_{ij}^2}}\int_{\mathbb{R}^2}\exp\lc \tilde{\rho}_{ij}\lk x^2+y^2\rk +2\bar{\rho}_{ij}xy \rc  \times\,& \exp\lc -\lk x^2+y^2 \rk  \rc dxdy\\
	=\,&\frac{2}{\sqrt{1-\rho_{ij}^2}}\cdot \frac{\pi}{{\sqrt{(1-\tilde{\rho}_{ij})^2-\bar{\rho}_{ij}^2}}} \,.
\end{align*}
Thus  to bound $B_{ij}$ we need to know the asymptotics of $\rho_{ij}$ and   $\bar \rho_{ij}$. First,
there exists a constant $C_{H}$ such that
\begin{align}\label{order}
\mathbb{E}[V_i^HV_j^H]&=\mathbb{E}\left[(B^{H}(t_{i+1})-B^{H}(t_i))(B^{H}(t_{j+1})-B^{H}(t_j))\right]\nonumber\\
&=\frac{1}{2}\left[(t_{i+1}-t_j)^{2H}-(t_{i+1}-t_{j+1})^{2H}-(t_i-t_j)^{2H}+(t_i-t_{j+1})^{2H}\right]\nonumber\\
&=\frac{(\Delta t)^{2H}}{2}\lk (i-j+1)^{2H}-2(i-j)^{2H}+(i-j-1)^{2H} \rk \\
&\leq \frac{(\Delta t)^{2H}}{2H(2H-1)} \cdot\lk (i-j-1)^{2H-2}\wedge 1 \rk \leq C_H \cdot(\Delta t)^{2H}\lk \mid i-j\mid ^{2H-2} \wedge 1 \rk\,.\nonumber
\end{align}
Hence,
\begin{equation}
0\le \rho_{ij}=\frac{\mathbb{E}[V_i^HV_j^H]}{\sigma_i \sigma_j}\le
C_H   \lk \mid i-j\mid ^{2H-2} \wedge 1 \rk    \quad \hbox{as $\mid i-j\mid \to \infty$}\,.
\end{equation}
Consequently,
\begin{equation}
\bar \rho_{ij}  \leq
  C\lk \mid i-j\mid ^{2H-2} \wedge 1 \rk\,,\quad \tilde  \rho_{ij}  \leq
  C \lk \mid i-j\mid ^{4H-4} \wedge 1 \rk   \quad \hbox{as $\mid i-j\mid \to \infty$}\,.
\end{equation}
Therefore, by the relation  $(1+x)^{\alpha}-1\asymp \alpha x$ as $x\to 0$,  we have
\begin{align*}
	B_{ij}\leq&C\lk {\frac{\pi}{1-2\tilde{\rho}_{ij}}+\frac{1}{{1-\rho_{ij}^2}}\cdot\frac{\pi}{\sqrt{1-4\bar{\rho}_{ij}^2}}-\frac{2}{\sqrt{1-\rho_{ij}^2}}\cdot \frac{\pi}{\sqrt{(1-\tilde{\rho}_{ij})^2-\bar{\rho}_{i,j}^2} }}\rk \\
	=&\frac{\pi}{\sqrt{1-\rho_{ij}^2}}\lk \frac{\sqrt{1-\rho_{ij}^2}}{1-2\tilde{\rho}_{ij}}+\frac{1}{\sqrt{1-\rho_{ij}^2}}\cdot \frac{1}{\sqrt{1-4\bar{\rho}_{ij}^2}}-\frac{2}{\sqrt{(1-\tilde{\rho}_{ij})^2-\bar{\rho}_{ij}^2}} \rk \\
	=&\frac{\pi}{\sqrt{1-\rho_{ij}^2}}\lk \lc\frac{\sqrt{1-\rho_{ij}^2}}{1-2\tilde{\rho}_{ij}}-1\rc+\lc\frac{1}{\sqrt{1-\rho_{ij}^2}}\cdot \frac{1}{\sqrt{1-4\bar{\rho}_{ij}^2}}-1\rc+2-\frac{2}{\sqrt{(1-\tilde{\rho}_{ij})^2-\bar{\rho}_{ij}^2}} \rk \\
	\leq & C\lk \mid 2\tilde{\rho}_{ij}-\frac 12 \rho_{ij}^2 \mid + \mid \frac 12(\rho_{ij}^2+4\bar{\rho}_{ij}^2-4\rho_{ij}^2
\bar{\rho}_{ij}^2) \mid +\lvert \frac 12(\tilde{\rho}_{ij}^2-\bar{\rho}_{ij}^2-2\tilde{\rho}_{ij}) \rvert \rk\\
\leq & C \mid i-j\mid^{4H-4}\,,
\end{align*}
when $\mid i-j\mid $ is sufficiently  large.
 As a result, we have
\begin{align}\label{eq.Cov_bd}
\mathrm{Cov}(Y_i,Y_j)\leq C\lk \mid i-j \mid ^{2H-2}\wedge 1\rk\,.
\end{align}
This completes the proof of  \eqref{Ineq.Cov} and  hence  we finish the proof of part (iii) of Theorem \ref{Thm_1}.
\end{proof}

\subsection{ Brownian motion case}
In this section we consider the case when $H=1/2$, namely,   Brownian motion   $B^H=B$. The equation \eqref{eq.SDE_FB} becomes
\begin{align}\label{eq.SDE_B.P}
dX(t)=-\lambda\kappa t^{\kappa-1}  X(t)dt+\mu X(t)\circ dB(t)\,,~~~X(0)=X_0\,,
\end{align}
  where $B$ is standard Brownian motion, $\lambda,\mu\in \mathbb{R}$ and $\kappa\geq 2H=1$. Here, we   assume that $X_0\neq 0$ with a positive  probability   and $\mid X_0\mid $ is square integrable.  We have $X(t)=X_0\exp(\lambda t^{\kappa}+\mu B(t))$ and
\begin{equation}\label{eq.L2_B.P}
	\mathbb{E}\mid X(t)\mid ^{2}=\mathbb{E}\mid X_0\mid ^2\exp\left(2(-\lambda t^{\kappa}+\mu^2 t)\right)\,.
\end{equation}
So the solution is stable if
 (i) $\kappa>1$ and $\lambda>0$ or  (ii) $\kappa=1$ and $-\lambda+\mid \mu\mid ^2<0$. Otherwise, the solution of \eqref{eq.SDE_B.P} is unstable.

For the above equation we can write  \eqref{STM_linear}  as
\begin{equation}\label{B.3}
X_{n+1} =\,\left(\frac{1-\kappa(1-\theta)\lambda (t_n)^{\kappa-1} \Delta t}{1+\kappa\theta\lambda (t_{n+1})^{\kappa-1}\Delta t}
+\frac{\mu V_n}{1+\kappa\theta\lambda (t_{n+1})^{\kappa-1}\Delta t}\right)X_{n} \,,
\end{equation}
with $t_{n}=n\cdot \Delta t$ and   $V_{n}=B(t_{n+1})-B(t_{n})$. Notice that $V_{n}$'s are mutual independent.  The  equation  \eqref{B.3} can also be rewritten as follows
\begin{align}\label{B.q4}
X_{n+1}=X_{0}\prod_{k=1}^{n}Z_{k}(\Delta t)=X_{0}\prod_{k=1}^{n}\left(\alpha_{k}
+\beta_{k}V_{k}\right)\,.
\end{align}
where
\begin{empheq}[left=\empheqlbrace]{align}
&\alpha_{n}:=\alpha_{n}(\theta,\lambda,\Delta t)=\, \frac{1-\kappa(1-\theta)\lambda (t_n)^{\kappa-1} \Delta t}{1+\kappa\theta\lambda (t_{n+1})^{\kappa-1}\Delta t}=\frac{1-\kappa(1-\theta)\lambda n^{\kappa-1} \Delta t^{\kappa}}{1+\kappa\theta\lambda (n+1)^{\kappa-1}\Delta t^{\kappa}}\,, \nonumber\\
&\beta_{n}:=\beta_{n}(\theta,\lambda,\mu,\Delta t)=\, \frac{\mu}{1+\kappa\theta\lambda (t_{n+1})^{\kappa-1}\Delta t}=\frac{\mu}{1+\kappa\theta\lambda (n+1)^{\kappa-1}\Delta t^{\kappa}}\,.
\nonumber
\end{empheq}
 Obviously,  we have the following.
 \begin{itemize}
   \item If $\kappa>1$, for every fixed $\Delta t>0$, $\lambda$ and $\mu$ (even for $\lambda>0$)
$$
\lim_{n\to \infty}\alpha_{n}=-\frac{1-\theta}{\theta},~~~~~~
\lim_{n\to \infty}\beta_{n}=0.
$$
Therefore, we have
\begin{align*}
	\EE[\mid X_{n+1}\mid ^2]=\,&\EE[\mid X_0\mid ^2]\prod_{k=1}^{n}\EE\left[\mid \alpha_{k}
+\beta_{k}V_{k}\mid ^2\right] \\
=\,&\EE[\mid X_0\mid ^2]\prod_{k=1}^{n}\lk \alpha_k^2+\beta_k^2 \cdot \Delta t \rk \asymp \lc\frac{1-\theta}{\theta}\rc^{2n}\to\begin{cases}
		0\,, &\text{if~} \frac 1 2<\theta\leq 1\,;\\
		\infty\,, &\text{if~} 0\leq \theta< \frac 1 2\,.
	\end{cases}
\end{align*}
   \item If $\kappa=1$, \eqref{eq.SDE_B.P} is reduced  to the standard stochastic test equation (see also \cite{Komori2012SROCK}). Then
$$
\alpha_{n}=\bar{\alpha}=\frac{1-(1-\theta)\lambda\Delta t}{1+\theta\lambda\Delta t}\,,\qquad
\beta_{n}=\bar{\beta}=\frac{\mu}{1+\theta\lambda\Delta t}\,.
$$
Thus,
$$
\mathbb{E}\mid X_{n+1}\mid ^2=\mathbb{E}(\bar{\alpha}+\bar{\beta}\cdot V_n)^2\mathbb{E}\mid X_n\mid ^2\,.
$$
In this sense, the numerical stability (or non-stability) depends on the condition
\begin{align*}
	&\bar{\alpha}^2+\bar{\beta}^2\cdot \Delta t<1 ~(\text{or }>1)\,,\\
	\Leftrightarrow\quad&(1-2\theta)\lambda^2 \Delta t+(-2\lambda+\mid \mu\mid ^2)<0 ~(\text{or }>0)\,.
\end{align*}
 \end{itemize}
Now, we can summarize the discussion above as the following proposition:
\begin{Pro}
	For the test equation \eqref{eq.SDE_B.P} and the STM \eqref{B.3}, we have
	\begin{enumerate}
	\item[(i)]	When $\kappa>1$, for any fixed $\lambda,~\mu$, then the STM \eqref{B.3} is mean square stable for the test equation \eqref{eq.SDE_B.P} if $\frac 12<\theta\leq 1$ and is \emph{not} mean square stable if $0\leq \theta\leq \frac 12$\,;
	\item[(ii)] When $\kappa=1$ and $-2\lambda+\mid \mu\mid ^2<0$, then the STM \eqref{B.3} is mean square stable for the test equation \eqref{eq.SDE_B.P} if either $\frac 12\leq \theta\leq 1$ for all $\Delta t>0$ or  $0\leq \theta< \frac 12$ for $\Delta t$ satisfying
	\begin{equation*}
		0<\Delta t<\frac{2\lambda-\mid \mu\mid ^2}{(1-2\theta)\lambda^2}\,;
	\end{equation*}
	\item[(iii)] When $\kappa=1$, $-2\lambda+\mid \mu\mid ^2>0$ and $0\leq \theta< 1/2$, then the STM \eqref{B.3} is \emph{not} mean square stable for the test equation \eqref{eq.SDE_B.P} for all $\Delta t>0$\,.
	\end{enumerate}
\end{Pro}

\section{STM: Mean square nonlinear stability analysis}
In this section, we shall study the
$p$-th moments stability and the numerical stability of the solution to the general SDEs driven by fBm.

\subsection{Mean square stability}
{The existence and uniqueness problems of \eqref{gene_FSDE} have been studied extensively in the last two decades. For precise results, we refer \cite{FZ2021} and the references therein.  
Beyond the well-posedness, as we mentioned in Section 1, it seems  too complicated to find the long time asymptotic behavior of \eqref{gene_FSDE}.} 
To the best of our knowledge, there is few results on the convergence of $\EE[\mid X(t)\mid ^2]$ when $t$ goes to infinity. Thus, we focus on the following simplified SDEs with $g(t,X(t))=c(t)X(t)$ under the assumption \eqref{LG_f} and \eqref{LG_g} 
\begin{equation}\label{Sim_SDE}
	d X(t)=f(t,X(t))dt+c(t)X(t)dB^{H}(t)\,.
\end{equation}
In this case, \eqref{LG_g}  means $\mid c(t)\mid \leq \mu$ for some $\mu>0$.

\begin{Theo}
	Let $X(t)$ be the solution to SDE \eqref{Sim_SDE}, and let $p$, $\kappa$ and $\lambda$, $\mu$ satisfy
	\begin{equation}\label{Cond_MS.q}
	 {\text{(i)}\ \kappa>2H \text{~and~} \lambda>0 \quad\text{~or~}\quad  \text{(ii)}\ \kappa=2H \text{~and~} -\lambda+\frac{p}{2}\mu^2<0\,.}
\end{equation}
	 If $f(t,x)$ in \eqref{Sim_SDE} satisfies  \eqref{Lip_f} in Assumption \ref{gene-assum}
 and $c(t)$ satisfies $\mid c(t)\mid \leq \mu$ for $\mu>0$.  Then for  any   $p\geq 1$,
  $\EE\lvert X(t)\rvert^p\to 0$ as $t\to\infty$ under the condition \eqref{Cond_MS}.
\end{Theo}
\begin{Rem}
	When $p=2$, i.e. the mean square stable case, the condition \eqref{Cond_MS.q} coincides with \eqref{Cond_MS}.
\end{Rem}

\begin{proof}
We can assume that  $p\geq 2$ is  an even positive number.
Denote  $F_t=\exp[-\int_0^t c(s) dB^H_s]$ and $Y_t=F_t\cdot X(t)$.  Then by the chain rule  formula (e.g. \cite[Proposition 2.7]{Hu13} or \cite[Lemma 2.7.1]{Mishura2008book}) we have
\begin{align*}
	\frac{d}{dt} Y_t=F_t\cdot f(t,X(t)) = F_t\cdot f(t,(F_t)^{-1}Y_t) \,.
\end{align*}
Note that it is a deterministic ordinary differential equation for the function $t\to Y_t(\omega)$ for every $\omega\in\Omega$. Then by the condition \eqref{Lip_f} in Assumption \ref{gene-assum}, we   get
\begin{align*}
	\frac{d}{dt} Y_t^p &=pY_t^{p-1}\frac{dY_t}{dt}\\
&=pY_t^{p-1} (F_t)^{-1} F^2_t f(t,(F_t)^{-1}Y_t) \leq -p  \lambda\kappa t^{\kappa-1} \cdot Y_t^p\,.
\end{align*}
Thanks to Gronwall's inequality, we have
\begin{align*}
	Y_t^p \leq&\ Y_0^p \exp\lc-p\lambda\kappa\int_0^t s^{\kappa-1}ds \rc=X_0^p \exp(-p\lambda t^{\kappa})\,,
\end{align*}
and
 \begin{align*}
 	X(t)^p \leq&\ X_0^p \exp\lc -p\lambda t^{\kappa} +p \int_0^t  c(s) dB^H_s \rc\,.
 \end{align*}
Therefore, letting $C_H=H(2H-1)$,   we have 
\begin{align*}
	\EE [X(t)^p] \leq&\ C  \exp\lc -p\lambda t^{\kappa} + p^2 {C_H} \int_0^t \int_0^t c(s)\mid s-r\mid ^{2H-2} c(r)drds \rc \\
	\leq&\ C \exp\lc -p\lambda t^{\kappa} + \frac{p^2}{2} \mu^2 t^{2H} \rc\,.
\end{align*}
So we have that $\EE [X(t)^p]$ converges to $0$ under the condition \eqref{Cond_MS.q}. 
\end{proof}

\subsection{Numerical stability}
{ For the numerical stability of SDEs driven by fBm, we consider a more general
diffusion coefficient. More precisely, instead of $g(t,X)=c(t)X$ 
we allow the  diffusion term $g$ to be generally  nonlinear 
satisfying \eqref{LG_g}.   
We hope this will shed    light to  the 
stability of the original solution.}

 Now, we give the proof of Theorem \ref{Thm_2}.

\begin{proof}[Proof of Theorem \ref{Thm_2}]
From STM \eqref{STM}, we have
\begin{align}\label{STM_X_n+1}
X_{n+1}-\theta f(t_{n+1},X_{n+1})\Delta t= X_n+ (1-\theta)f(t_n,X_n)\Delta t + g(t_n,X_n)V_n^H \,.
\end{align}
By the condition \eqref{Lip_f}, we have
$$
\mid f(t,X)\mid ^2\geq (\lambda \kappa t^{\kappa-1})^2\mid X\mid ^2\,.
$$
We  bound  the square of left hand side of \eqref{STM_X_n+1} as
\begin{align}\label{left-formula}
\begin{split}
\mid X_{n+1}-&\theta f(t_{n+1},X_{n+1})\Delta t\mid ^2\\
& = (X_{n+1})^2+ \theta^2 \mid f(t_{n+1},X_{n+1})\mid ^2 (\Delta t)^2 -2 \theta \Delta t X_{x+1}f(t_{n+1},X_{n+1}) \\
&\geq (X_{n+1})^2 + (\theta~ \lambda~ \kappa~ t_{n+1}^{\kappa-1}~\Delta t)^{2} (X_{n+1})^2
+ 2 ~\theta~\lambda~\kappa~t_{n+1}^{\kappa-1}~\Delta t~ (X_{n+1})^2\\
&= (X_{n+1})^2 [1+\theta~ \lambda~ \kappa~ t_{n+1}^{\kappa-1}~\Delta t]^2\,.
\end{split}
\end{align}

\textbf{Step 1 ($\theta=1$):} With the condition \eqref{LG_g}, it is clear that the square of right hand side of \eqref{STM_X_n+1} can be bounded by
\begin{align}\label{right-formula_1}
\mid X_n+  g(t_n,X_n)V_n^H\mid ^2
\leq 2\lk (X_{n+1})^2+ \mu^2 X_n^2 (V_n^H)^2\rk \,.
\end{align}
Therefore, we have from \eqref{left-formula} and \eqref{right-formula_1}
\begin{align}\label{X_n+1.rec_1}
\mid X_{n}\mid ^2 \leq 2 \left[\alpha_n^2+\beta_n^2\cdot (V_n^H)^2\right] \mid X_{n-1}\mid ^2&\leq 2^{n} \prod_{j=1}^{n}\left[\alpha_j^2 +\beta_j^2(V_j^H)^2 \right] X_0^2\,,
\end{align}
where
\begin{equation*}
	\alpha_{n} = \frac{ 1}
{1+ \lambda\cdot \kappa(t_{n})^{\kappa-1}\Delta t}\,,\qquad
\beta_n = \frac{\mu }
{1+ \lambda\cdot \kappa(t_{n})^{\kappa-1}\Delta t}\,.
\end{equation*}
Let us rewrite $\prod_{j=1}^{n}\left[\alpha_j^2 +\beta_j^2(V_j^H)^2 \right]$ in \eqref{X_n+1.rec_1} as
$$
\prod_{j=1}^{2n} Z_j= \prod_{j=1}^{n} (\beta_jV_j^H+\iota \alpha_j) (\beta_j V_j^H-\iota \alpha_j)\,,
$$
with $\iota$ being the imaginary number  and
$$
Z_{2j-1}=\beta_{j}V_j^H+ \iota\alpha_j\,,\qquad\qquad Z_{2j}=\beta_{j}V_j^H- \iota\alpha_j\,.
$$
 Applying  Lemma \ref{Polarization} with $s_1=\cdots=s_{2n}=1$, $s=\sum_{j=1}^{2n}s_j=2n$, one has
 \begin{align}\label{gene_polar}
 \prod_{j=1}^{2n} Z_j=\frac{1}{2n!}\sum_{v_1=0}^{1}\cdots \sum_{v_{2n}=0}^{1}
 {1 \choose v_1}\cdots {1 \choose v_{2n}}(-1)^{\sum_{j=0}^{2n+1}v_j}\left[\sum_{j=1}^{2n}h_j Z_j \right]^{2n}\,,
 \end{align}
 where $h_j=\frac 12-v_j=\frac{(-1)^{v_j}}{2}$. Note that
\begin{align*}
\sum_{j=1}^{2n}h_jZ_j&=\sum_{j~ odd}h_j(\beta_{j} V_j^H+\iota \alpha_j) + \sum_{j~ even}h_j(\beta_j V_j^H- \iota\alpha_j)\\
&= \sum_{j=1}^{2n}h_j \beta_j\cdot V_j^H + \iota\left(\sum_{j~ odd}h_j\alpha_j -  \sum_{j~ even}h_j\alpha_j\right)\,.
\end{align*}
Thus, we get that
\begin{align*}
\mid \sum_{j=1}^{2n}h_jZ_j \mid^{2n} &= \left[\lft( \sum_{j=1}^{2n}h_j \beta_j\cdot V_j^H\rht)^2 + \lft( \sum_{j~ odd}h_j\alpha_j -  \sum_{j~ even}h_j\alpha_j \rht)^2\right]^{n}\\
&\leq 2^n \left[\lft( \sum_{j=1}^{2n}h_j \beta_j\cdot V_j^H \rht)^{2n}+
 \lft( \sum_{j~ odd}h_j\alpha_j -  \sum_{j~ even}h_j\alpha_j \rht)^{2n}\right]\,.
\end{align*}
Therefore, taking expectation on both sides  of \eqref{X_n+1.rec_1}, we obtain
\begin{align*}
\EE [\mid X_{n}\mid ^2]&\leq 2^{n}\cdot \EE \lk \bigg\vert \prod_{j=1}^{2n} Z_j \bigg\vert  \rk\\
&\leq \frac{2^{n}}{2n!}\sum_{v_1=0}^{1}\cdots \sum_{v_{2n}=0}^{1}
 {1 \choose v_0}\cdots {1 \choose v_{2n}}\EE\left[ \bigg\vert \sum_{j=1}^{2n}h_j Z_j \bigg\vert^{2n} \right]\\
&\leq \frac{2^{n}}{2n!} \sum_{v_1=0}^{1}\cdots \sum_{v_{2n}=0}^{1} \frac{1}{2^{n}}
\left[\EE \lc\sum_{j=1}^{2n}(-1)^{v_j} \beta_j\cdot V_j^H\rc^{2n}+
 \lc \sum_{j=0}^{2n}\alpha_j \rc^{2n}\right]\\
&=: \widetilde{I_1} + \widetilde{I_2} \,.
\end{align*}

 Denote $R=\sum_{j=1}^{2n}(-1)^{v_j} \beta_j\cdot V_j^H$.  Then $R$ is a Gaussian
 random variable with mean zero and variance $\sigma_{R}^2$ given by 
$$
\sigma_{R}^2=\EE\left(\sum_{j=1}^{2n}(-1)^{v_j} \beta_j\cdot V_j^H \right)^2 \leq C\cdot n^{2+2H-2\kappa} 
$$
 (see the computation in the appendix A).  Thus, we have 
\begin{align*}
	\EE\left(\sum_{j=1}^{2n}(-1)^{v_j} \beta_j\cdot V_j^H \right)^{2n}
 =&\ 2^{n} \Gamma(n+1/2)\cdot (\sigma_{R})^{2n} \\
 \leq&\ C^n 2^{n} \Gamma(n+1/2)\cdot n^{(2+2H-2\kappa)n} \,.
\end{align*}
By   Stirling's formula, we further have
\begin{align*}
\widetilde{I_1}&\leq \frac{2^{n}}{2n!} \cdot \frac{2^{2n}}{2^{n}} \cdot
 \EE\left(\sum_{j=1}^{2n}(-1)^{v_j} \beta_j\cdot V_j^H \right)^{2n}\\
&\leq \frac{4^{n}C^{n}}{2n!}\cdot  \Gamma(n+1/2)\cdot n^{n(2+2H-2\kappa)} \\
& \asymp \frac{1}{\blc\frac{2n}{e} \brc^{2n}}\cdot n^{n(2+2H-2\kappa)} \blc\frac{n-\frac 12}{e} \brc^{n-\frac 12}\\
& \asymp C^n \cdot n^{n(1+2H-2\kappa)} \to 0\,, 
\end{align*}
as $n\to \infty$ since $\kappa\geq 2H>1$.
For the term $\widetilde{I_2}$, as $n\to \infty$, we also have
\begin{align*}
\widetilde{I_2}&\leq \frac{2^{n}}{2n!} \cdot \frac{2^{2n}}{2^{n+1}} \lc\sum_{j=1}^{2n}\alpha_j \rc^{2n} \leq \frac{C^{n}\cdot n^{2(2-\kappa)n}}{\sqrt{2\pi(2n)}\blc\frac{2n}{e} \brc^{2n}}\asymp C^n\cdot n^{2(1-\kappa)n}\to 0 \,.
\end{align*}

\textbf{Step 2 ($\theta<1$):} Similar to \eqref{right-formula_1}, it follows with additional condition \eqref{LG_f} that the right hand side of \eqref{STM_X_n+1} can be bounded by
\begin{align}\label{right-formula_2}
\begin{split}
\mid X_{n}+&(1-\theta)f(t_n,X_n)\Delta t+ g(t_n,X_n)V_n^H\mid ^2\\
& \leq 3\left(X_n^2 + \mid (1-\theta)f(t_n,X_n)\Delta t\mid ^2 + \mid g(t_n,X_n)V_n^H\mid ^2 \right) \\
&\leq 3\left(X_n^2 + (1-\theta)^2 (\bar{\lambda}~\kappa~ t_{n+1}^{\kappa-1}~\Delta t)^2 X_n^2
 + \mu^2 X_n^2 (V_n^H)^2 \right)\,.
\end{split}
\end{align}
Combining \eqref{left-formula} and \eqref{right-formula_2}, we have
\begin{align*}
[1+ \theta~\lambda~\kappa~t_{n+1}^{\kappa-1}~\Delta t]^2 (X_{n+1})^2  \leq
3 \left(1+ (1-\theta)^2(\bar{\lambda}~\kappa~t_{n+1}^{\kappa-1}~\Delta t)^2 +\mu^2 (V_n^H)^2\right) (X_n)^2\,. 
\end{align*}
We can the above inequality as 
\begin{align*}
(X_{n})^2 \leq 3 \left[\alpha_n^2+\beta_n^2\cdot (V_n^H)^2\right] (X_{n-1})^2\,,
\end{align*}
where  $\alpha_n$ and $\beta_n$ are given  by
\begin{equation}\label{def_al.beta}
	\alpha_{n}^2 = \frac{ 1+ (1-\theta)^2(\bar{\lambda}~\kappa~t_{n}^{\kappa-1}~\Delta t)^2  }
{[1+ \theta~\lambda~\kappa~t_{n}^{\kappa-1}~\Delta t]^2}\,,\qquad
\beta_n^2 = \frac{\mu^2}
{[1+ \theta~\lambda~\kappa~t_{n}^{\kappa-1}~\Delta t]^2}\,.
\end{equation}
Therefore, 
\begin{align}\label{X_n+1.rec_2}
X_{n}^2 &\leq 3^{n} \prod_{j=1}^{n}\left(\alpha_j^2 +\beta_j^2(V_j^H)^2 \right) X_0^2 \nonumber \\
&= 3^{n} \prod_{j=1}^{n} (\beta_jV_j^H+\iota \alpha_j) (\beta_jV_j^H-\iota \alpha_j)\,.
\end{align}
Thus by the same procedure as in  \textbf{Step 1 ($\theta=1$)}, taking expectation on both sides of \eqref{X_n+1.rec_2} gives
\begin{align*}
\EE [\mid X_{n}\mid ^2]&\leq 3^{n}\cdot \EE \lk  \bigg\vert \prod_{j=1}^{2n} Z_j \bigg\vert  \rk\\
&\leq \frac{3^{n}}{2n!} \sum_{v_1=0}^{1}\cdots \sum_{v_{2n}=0}^{1} \frac{1}{2^{n}}
\left[\EE \lc\sum_{j=1}^{2n}(-1)^{v_j} \beta_j\cdot V_j^H \rc^{2n}+
 \lc \sum_{j=1}^{2n}\alpha_j \rc^{2n}\right]\\
&=: \widetilde{I_3} + \widetilde{I_4} \,.
\end{align*}

 We can prove that the term $\widetilde{I_3}$ converges to $0$ by the same technique used for the term $\widetilde{I_1}$ in  \textbf{Step 1 ($\theta=1$)}.
For the term $\widetilde{I_4}$, we further have
\begin{align*}
\widetilde{I_4}&\leq \frac{3^{n}}{(2n)!} \cdot \frac{2^{2n}}{2^{n+1}} \blc\sum_{j=1}^{2n}\alpha_j \brc^{2n}\\
&\asymp \frac{6^{n}}{\sqrt{2\pi\cdot 2n}\blc\frac{2n}{e} \brc^{2n}} \cdot (2n)^{2n}
\left(\frac{1-\theta}{\theta}\cdot \frac{\bar{\lambda}}{\lambda} \right)^{2n}\,.
\end{align*}
Obviously, we should require that
\begin{align}\label{stability-condition}
\sqrt{6}\cdot e\cdot \frac{1-\theta}{\theta}\cdot \frac{\bar{\lambda}}{\lambda} <1
\end{align}
to ensure $\widetilde{I_4}\to 0$ as $n\to \infty$.  The inequality \eqref{stability-condition} is equivalent to
 $\theta > \frac{\sqrt{6}\cdot e\cdot \frac{\bar{\lambda}}{\lambda}}{\sqrt{6}\cdot e\cdot \frac{\bar{\lambda}}{\lambda}+1}
 \geq \frac{\sqrt{6}e}{\sqrt{6}e+1}\approx 0.87$ since $\bar{\lambda}\geq \lambda$.  This completes the proof of Theorem \ref{Thm_2}. 
\end{proof}

\section{Numerical Experiments}
We shall  carry simulations for the following three equations. 
%

\begin{Exa}\label{example1}
Consider the following linear SDEs driven by fBm
\begin{equation}\label{4.1}
dX(t)=\kappa\cdot\lambda\cdot t^{\kappa-1}  X(t)dt+\mu X(t)dB^{H}(t)\,,
\end{equation}
with initial value $X(0)=3$.

\begin{figure}[htb]
\centering
\includegraphics[scale=0.8]{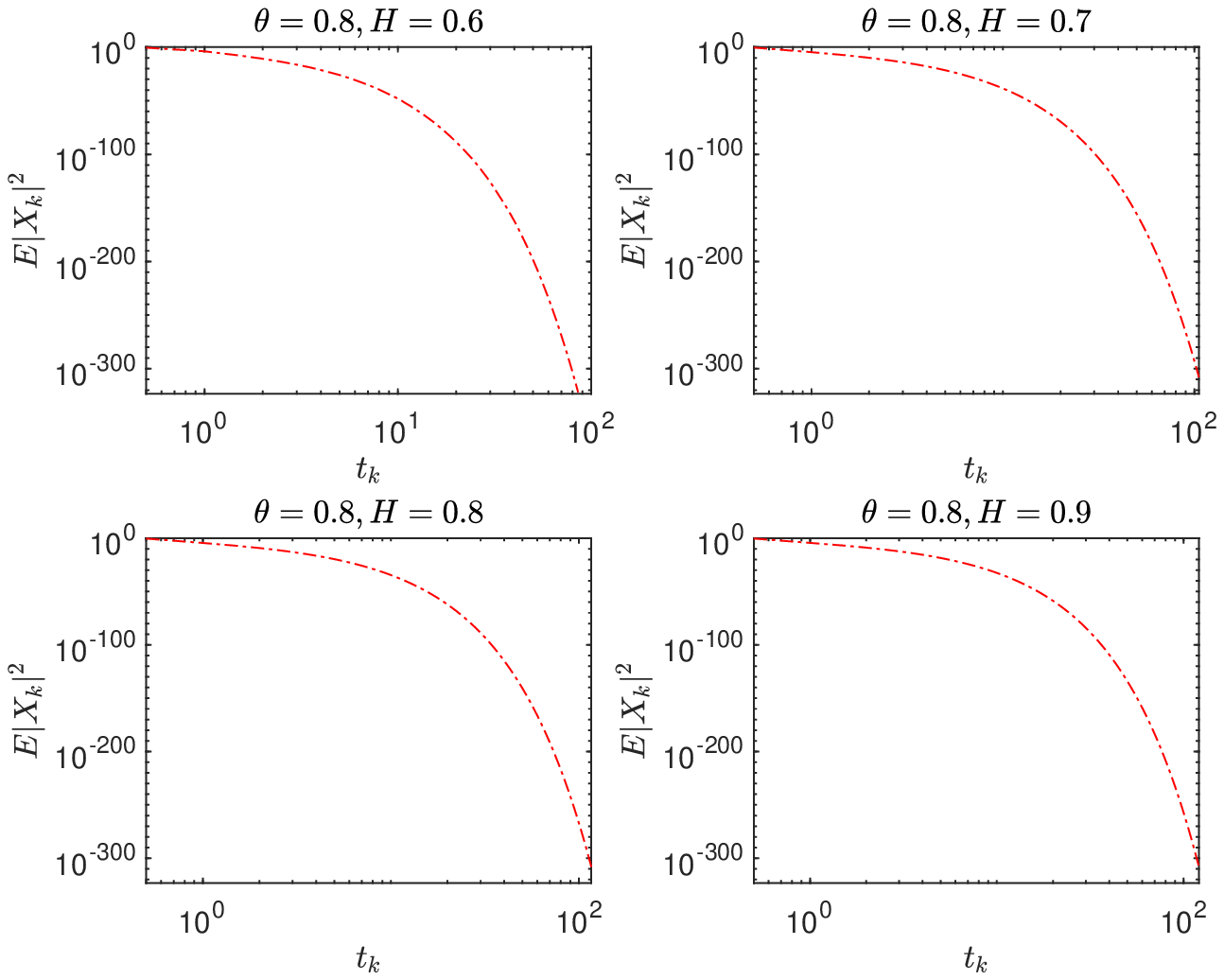}
\caption{Mean square of the STM \eqref{STM_linear} with $\theta=0.8$ for Example \ref{example1} with $\lambda=-9, \mu=2$.}
\label{stable}
\end{figure}

\begin{figure}[htb]
\centering
\includegraphics[scale=0.8]{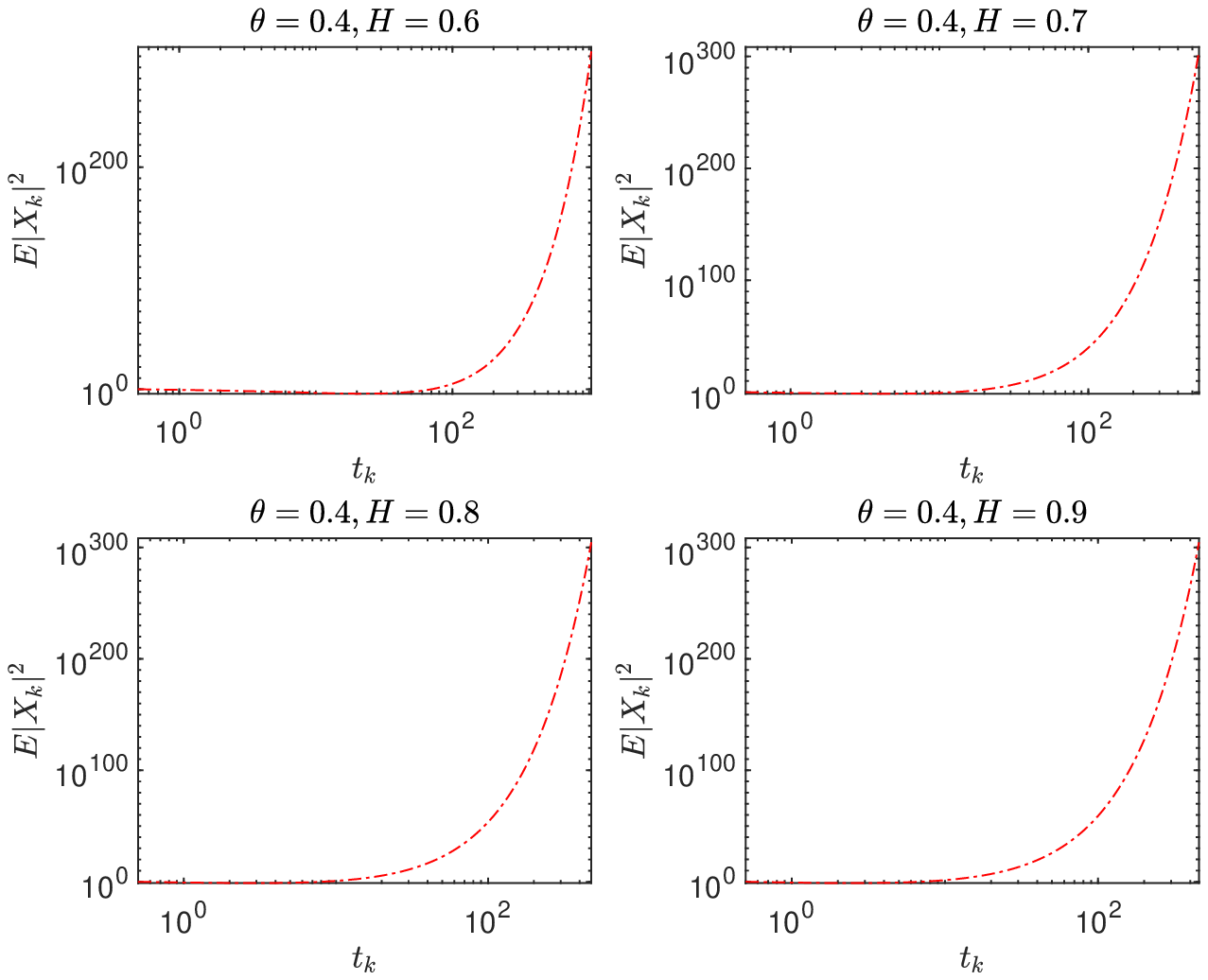}
\caption{Mean square of the STM \eqref{STM_linear} with $\theta=0.4$ for Example \ref{example1} with $\lambda=-9, \mu=2$.}
\label{unstable}
\end{figure}

\begin{figure}[htb]
\centering
\includegraphics[scale=0.8]{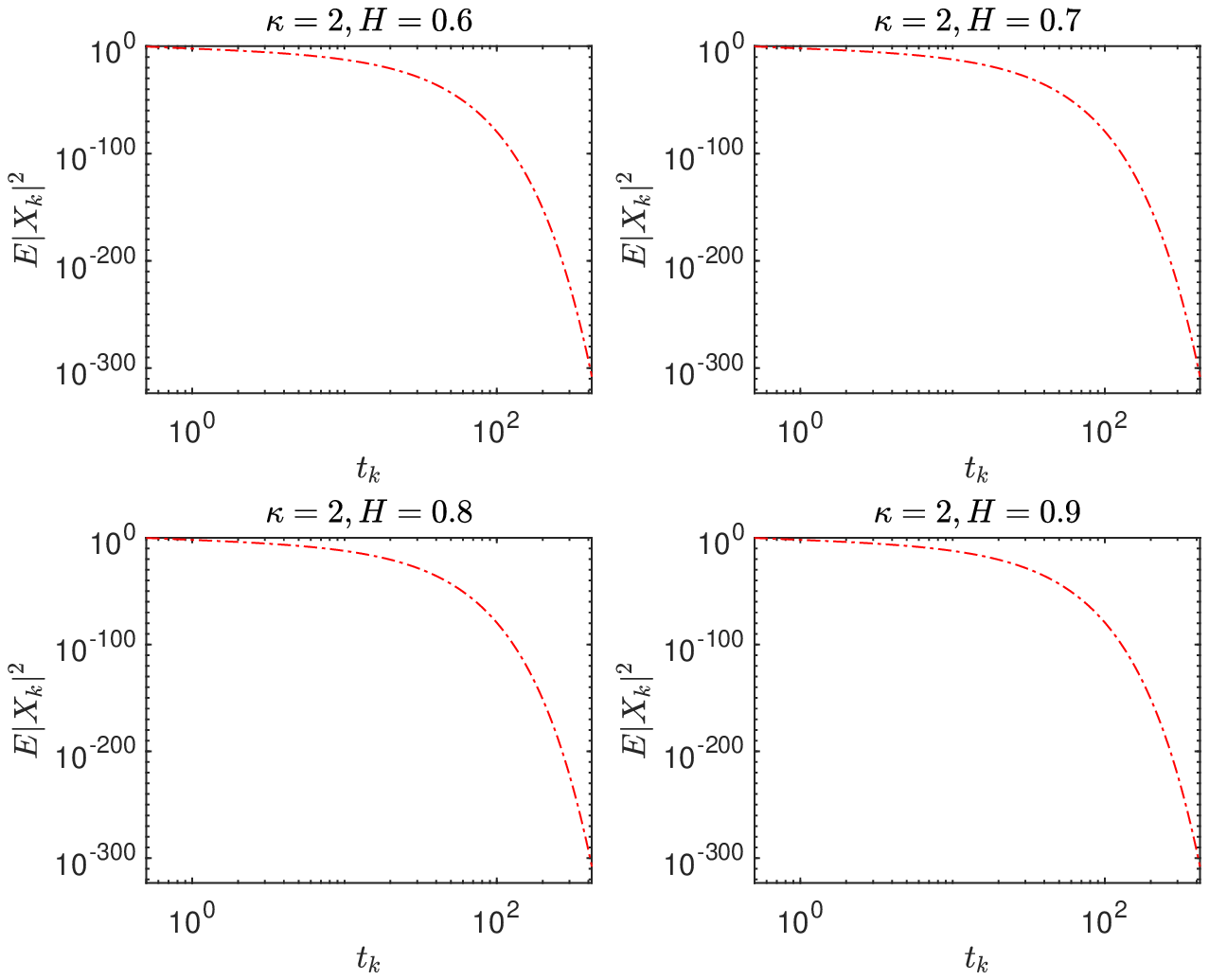}
\caption{Mean square of the STM \eqref{STM_linear} with $\theta=0.6$ for Example \ref{example1} with $\kappa=2, \lambda=-9, \mu=2$.}
\label{stable3}
\end{figure}

\begin{figure}[htb]
\centering
\includegraphics[scale=0.8]{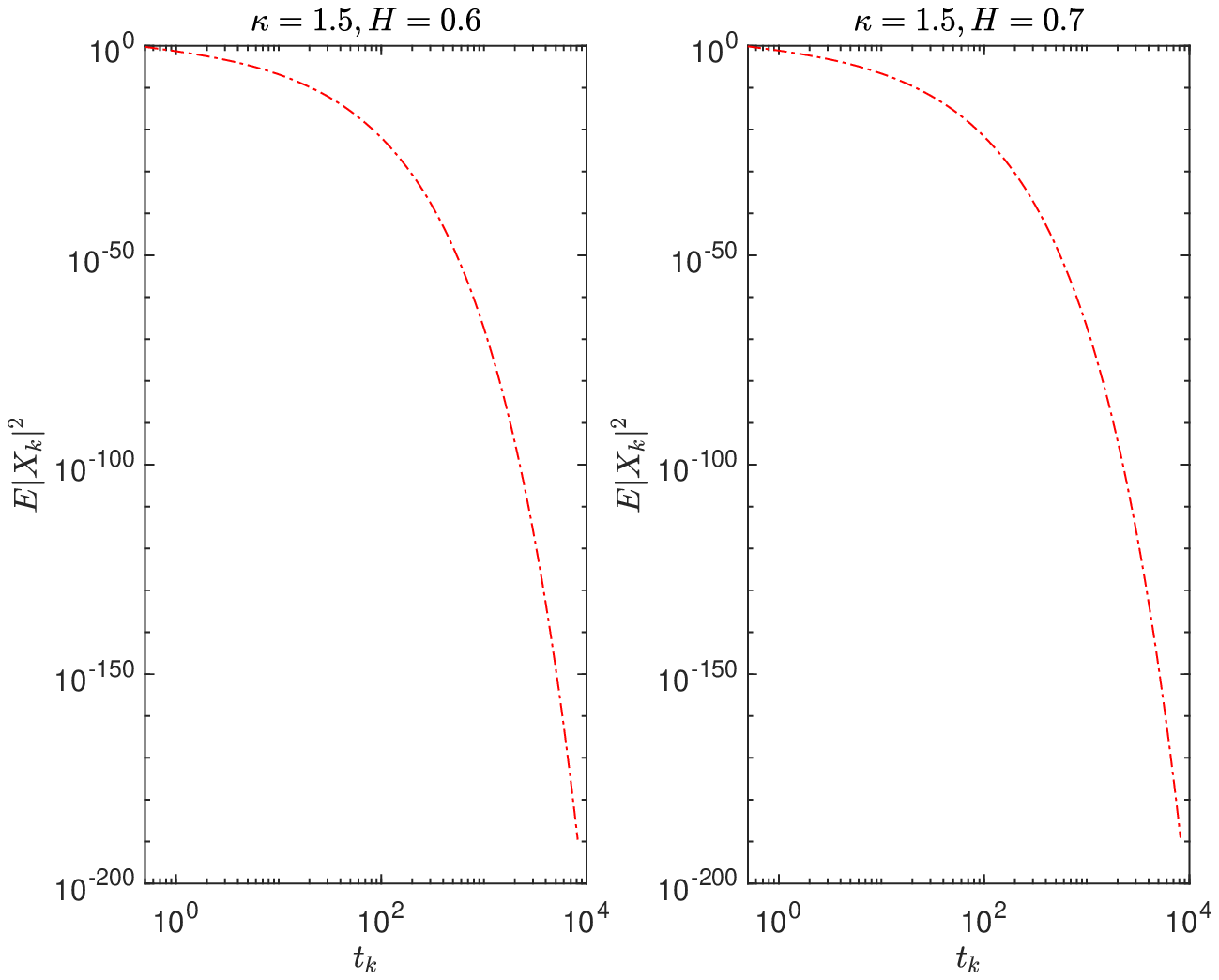}
\caption{Mean square of the STM \eqref{STM_linear} with $\theta=0.5$ for \eqref{4.1} with $\lambda=-9, \mu=2$.}
\label{3}
\end{figure}
\end{Exa}

\begin{Exa}\label{example2}
Consider the following nonlinear SDEs driven by fBm
\begin{equation}\label{4.2}
dX(t)=-\lambda\cdot \kappa\cdot t^{\kappa-1}  X(t)- X^3(t) dt+\mu X(t)dB^{H}(t)\,,
\end{equation}
with initial value $X(0)=3$.

\begin{figure}[htb]
\centering
\includegraphics[scale=0.8]{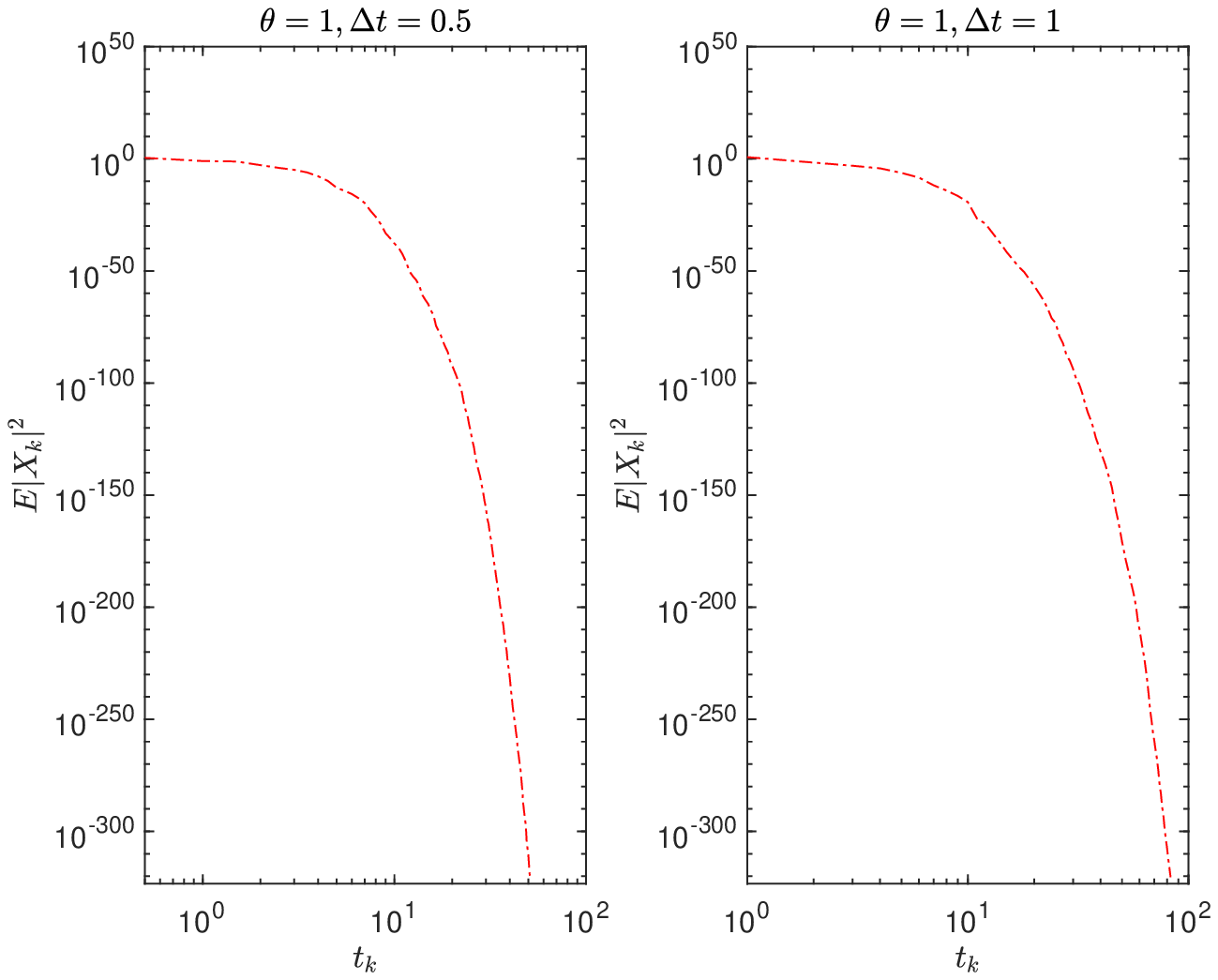}
\caption{Mean square of the STM \eqref{STM} with $\theta=1$ for Example \ref{example2}
 (left:$\Delta t=0.5$; right: $\Delta t=1$).}
\label{4}
\end{figure}

\end{Exa}

\begin{Exa}\label{example3}
Consider the following nonlinear SDEs driven by fBm
\begin{equation}\label{4.3}
dX(t)=-\lambda\cdot \kappa\cdot t^{\kappa-1}  X(t)- X^3(t) dt+ (X(t)+\sin(X(t))) dB^{H}(t)\,,
\end{equation}
with initial value $X(0)=3$.

\begin{figure}[htb]
\centering
\includegraphics[scale=0.8]{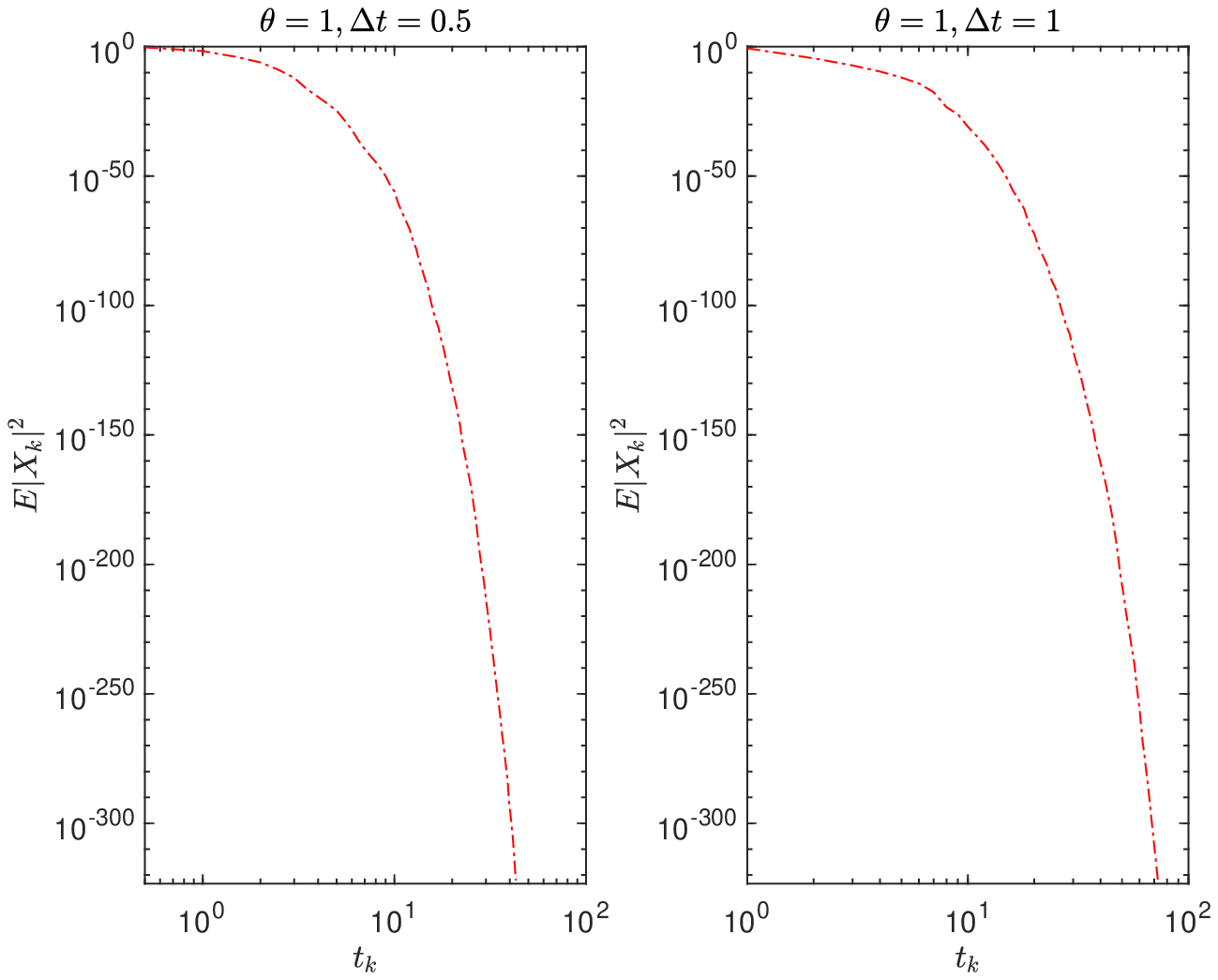}
\caption{Mean square of the STM \eqref{STM} with $\theta=1$ for Example \ref{example3}
 (left:$\Delta t=0.5$; right: $\Delta t=1$).}
\label{5}
\end{figure}

\end{Exa}

For the first test, we first fix $\lambda=-9, \mu=2, \kappa=2H$ and apply the stochastic theta  method with $\theta=0.8$ and $\theta=0.4$ for Example \ref{example1}
with different Hurst parameters $H$, respectively.
We take the stepsize $\Delta t=0.5$,
and the mean square of the numerical solutions over 5000
fBm samples are displayed in Figure \ref{stable} and Figure \ref{unstable} on a log-log scale, respectively.
 Besides, we also let $\kappa=2>\frac 32$, and the corresponding results of numerical solution are shown
 in Figure \ref{stable3}.
The expected
stable and unstable behaviors verify our theoretical results.

 For the test for Example \ref{example2}, we choose $\lambda=3$, $\kappa=2$,
 $\mu=4$ and $H=0.6$. It is easy to verify that the coefficients of the equation
 satisfy \eqref{Lip_f} and \eqref{LG_g} in Assumption \ref{gene-assum}. We take $\theta=1$ and
 $\Delta t= 0.5$ and $1$, respectively. The mean square of the
 numerical solutions are displayed in Figure \ref{4}. As expected, the
 stable behavior of the numerical solution is in agreement with Theorem \ref{Thm_2}.
 
 Moreover, we consider a more general diffusion term $g(t,X)=X(t)+\sin(X(t))$ 
 instead of $g(t,X)= c(t)X(t)$ with $c(t)\leq \mu$,
 in Example \ref{example3}. We take $\lambda=3$, $\kappa=2$,
  and $H=0.8$.  Figure \ref{5} depicts the mean square of the
 numerical solutions with $\Delta t= 0.5$ and $1$. The numerical 
 results illustrate that the STM with $\theta=1$ is also stable in this case. 
 We hope the numerical results can shed some light on the
 asymptotic property of the solution of the nonlinear equations 
 $
 dX(t) = f(t,X(t))dt + g(t,X(t))dB^H(t)\,.
 $

%

\section{Concluding remarks}
 This work first  focuses on the mean square stability of the stochastic theta method  for  the time non-homogeneous linear test equation driven by fBm,
$$
dX(t)=\lambda\kappa t^{\kappa-1}  X(t)dt+\mu X(t)dB^{H}(t)\,,~~~X(0)=3\,,
$$
  whose solution is
 stable in mean square sense.
For $\kappa\geq 2H$, it is proved that  the mean square
 A-stability of STM \eqref{STM_linear} is achieved for $\theta\geq \frac{\sqrt{3/2}\cdot e}{\sqrt{3/2}\cdot e+1}$,
  and the stochastic theta method cannot preserve the
 stability property of the test equation for $\theta<0.5$ in the sense of almost surely and mean square.
{Moreover, if $\kappa> 3/2$ and $\frac{1}{2}< \theta\leq 1$, then
the STM \eqref{STM_linear} is  mean square stable for the above  test equation.}
To illustrate our theoretical  results,
we give some simulation results for the equation \eqref{4.1}  with
$\lambda=-9, \mu=2, \Delta t=0.5 $
with different $H$ and $\theta$. The simulation results coincide well with our theoretical claims.
On the other hand,
we currently     are not  able to  use  our methods  to deal with the case  $\frac{1}{2}\leq \theta<\frac{\sqrt{3/2}\cdot e}{\sqrt{3/2}\cdot e+1}$ when $  2H\le \kappa\le 3/2$.  In this case, we simulate the equation \eqref{4.1}
with  $\lambda=-9, \mu=2,  \kappa= 1.4$ to
 test the stability by applying the  stochastic theta method with $\theta=0.5$, $\Delta t=0.5$ over the time interval
 $0 \leq t\leq 2^{13}$,
 the numerical results in Figure \ref{3}  show  that the method is still
 stable for $\theta=\frac 12$. Thus,
 we conjecture that  when  $2H\le \kappa\le 3/2$ and    $\frac{1}{2}\leq \theta<\frac{\sqrt{3/2}\cdot e}{\sqrt{3/2}\cdot e+1}$  the
 stochastic theta method is still mean square stable and this is our
   future research. { Finally, we also study the stability of the STM for nonlinear non-autonomous case
\[
 dX(t)=f(t,X(t))dt+g(t,X(t))dB(t)\,.
\]
 Under some conditions on the coefficients $f$ and $g$,
 it is proved that the STM method is stable when $\theta=1$. Moreover, under a stronger condition
on the coefficient of drift term $f$, the STM method is stable when
 $\theta> \frac{\sqrt{6}e\bar{\lambda}/\lambda}{\sqrt{6}e\bar{\lambda}/\lambda+1}$
 (where $\lambda$ and $\bar{\lambda}$ are defined \eqref{Lip_f} and \eqref{LG_f} in
 Assumption \ref{gene-assum}, respectively).}

\appendix
\section{Proof of $z^{2}=\frac{\mathsf{\tilde{m}}_n^2}{2\tilde{\sigma}_n^2}\gg n$}
In what follows, we   show that $z^{2}=\frac{\mathsf{\tilde{m}}_n^2}{2\tilde{\sigma}_n^2}\geq \frac{C\cdot n^{2}}{n^{2-2H}}\asymp n^{2H}\gg n$
 as $n\to \infty$.
 Note that
\begin{align*}
	\tilde{\mu}_n:=\tilde{\mu}_n(v_1,\cdots,v_n)=\,&\sum_{i=1}^n (1-v_i)\cdot \mu_i=\sum_{i=1}^n (1-v_i)\cdot \alpha_i
\,,\quad \mathsf{\tilde{m}}_n:=\tilde{\mu}_n(0,\cdots,0)=\sum_{i=1}^n \alpha_i\,, \\
	\tilde{\sigma}^2_n:=\tilde{\sigma}^2_n(v_1,\cdots,v_n)=\,&\EE\lk \Big\lvert \sum_{i=1}^n (1-v_i)\cdot \beta_i\cdot V_i^H  \Big\rvert ^2 \rk=\sum_{i,j=1}^{n} (1-v_i)(1-v_j)\cdot\beta_i\beta_j\cdot \EE[V_i^H V_j^H]\,.
\end{align*}
By  the property of fBm one can get with notation $\tilde{\beta}_j(\Delta t):=(1-v_j)\beta_{j}(\Delta t)$
\begin{align}\label{qq1.2}
 \tilde{\sigma}^2_n
 &=\sum\limits_{m,j=0}^{n}\tilde{\beta}_{m}(\Delta t)\tilde{\beta}_{j}(\Delta t)\mathbb{E}(V_{m}^{H}V_{j}^{H})\nonumber\\
 &=\frac{(\Delta t)^{2H}}{2}\sum\limits_{m,j=0}^{n}\tilde{\beta}_{m}(\Delta t)\tilde{\beta}_{j}(\Delta t)\lk \mid m-j+1\mid ^{2H}+\mid m-j-1\mid ^{2H}-2\mid m-j\mid ^{2H} \rk \,.
\end{align}
When $n$ and $\mid m-j\mid $ are large enough, we have
\begin{align*}
	\mid m-j+1\mid ^{2H}+&\mid m-j-1\mid ^{2H}-2\mid m-j\mid ^{2H} \\
	=&\mid m-j\mid ^{2H}\cdot \lk  \Big\lvert 1+\frac{1}{m-j} \Big\rvert^{2H}+ \Big\lvert 1-\frac{1}{m-j} \Big\rvert^{2H}-2 \rk\asymp \mid m-j\mid ^{2H-2}\,.
\end{align*}
 Therefore, we can bound \eqref{qq1.2} by
\begin{align}
	\tilde{\sigma}^2_n=&\frac{(\Delta t)^{2H}}{2}\sum\limits_{m,j=0}^{n}\tilde{\beta}_{m}(\Delta t)\tilde{\beta}_{j}(\Delta t)\lk \mid m-j+1\mid ^{2H}+\mid m-j-1\mid ^{2H}-2\mid m-j\mid ^{2H} \rk \nonumber\\
	\leq& C (\Delta t)^{2H}\sum\limits_{m\neq j}^{n} \tilde{\beta}_{m}(\Delta t)\tilde{\beta}_{j}(\Delta t) \mid m-j\mid ^{2H-2} \leq C (\Delta t)^{2H} \lc \sum\limits_{m=1}^{n} \mid \tilde{\beta}_{m}(\Delta t)\mid ^{\frac{1}{H}} \rc^{2H}\,,\label{qq1.3}
\end{align}
where we have used the discrete type Hardy-Littlewood-Sobolev inequality (see e.g. Theorem 381 in \cite{HLPIneq}) in the last step.

For any given $\Delta t>0$ (and $\lambda+\mid \mu\mid ^2<0$), one   observes  that from \eqref{qq1.3} 
\begin{align}
\lc\sum\limits_{m=0}^{n} \vert \tilde{\beta}_{m}(\Delta t)\vert^{\frac{1}{H}}\rc^{2H}\leq&\lc\sum_{m=0}^{n} \bigg\vert \frac{\mu}{1-\kappa\theta\lambda (m+1)^{\kappa-1}\Delta t^{\kappa}} \bigg\vert^{\frac{1}{H}}\rc^{2H} \nonumber\\
=&\lc \frac{\mu}{\kappa\theta \lambda (\Delta t)^\kappa}\rc^2 \lc\sum_{m=0}^{n} \bigg\vert \frac{1}{(m+1)^{\kappa-1}-\frac{1}{\kappa\theta \lambda (\Delta t)^\kappa}} \bigg\vert^{\frac{1}{H}}\rc^{2H} \nonumber\\
\leq & C  \lc\frac{\mu}{\kappa\theta \lambda (\Delta t)^\kappa}\rc^2 \cdot n^{2(1-\kappa)+2H}= \lc\frac{\mu}{\kappa\theta \lambda (\Delta t)^\kappa}\rc^2 \cdot n^{2+2H-2\kappa}\,.
\end{align}
Thus, $\tilde{\sigma}^2_n\leq C\cdot n^{2+2H-2\kappa} (\leq n^{2-2H}) \to 0$ as $n\to \infty$. Consequently,
$$
z^{2}=\frac{\mathsf{\tilde{m}}_n^2}{2\tilde{\sigma}_n^2}\geq  \frac{C\cdot n^{2}}{n^{2-2H}}\left(\frac{1-\theta}{\theta}\right)^{2}\asymp  (n+1)^{2H}\,.
$$

\section{Confluent Hypergeometric Functions}
\setcounter{equation}{3}

In this section, we gather some important properties of Kummer's confluent hypergeometric functions $\Phi(a,b,z)$ that are used in the main body of this work. The reader can find more details in Chapter 13 of \cite{NIST}.  Kummer's confluent hypergeometric functions $\Phi(a,b,z)$ is defined as
\begin{equation}\label{C.Kummer_def}
	\Phi(a,b,z)=\sum_{k=0}^{\infty}\frac{(a)_k}{(b)_k k!}z^k=1+\frac{a}{b}z+\frac{a(a+1)}{b(b+1)2!}z^2+\cdots\,.
\end{equation}
The following identity is called Kummer's transformation (see e.g.  13.2.29 in \cite{NIST})
\begin{equation}\label{C.Kummer_trans}
	\Phi(a,b,z)=e^z \Phi(b-a,b,-z)\,.
\end{equation}
A differentiation formula related to $\Phi(a,b,z)$ is helpful to us (13.3.20 in \cite{NIST}):
\begin{equation}\label{C.Kummer_diff}
	\frac{d^n}{dz^n}[e^{-z}\Phi(a,b,z)]=(-1)^n \frac{(b-a)_n}{(b)_n}\Phi(a,b+n,z)\,.
\end{equation}
Kummer's confluent hypergeometric functions $\Phi(a,b,z)$ can be represented by  the so-called parabolic cylinder functions $U(a,z)$ (13.6.14 and 13.6.15 in \cite{NIST}):
\begin{align}
	\Phi(a/2+1/4,1/2,z^2/2)=&\frac{2^{\frac a2-\frac 34}\Gamma(\frac a2+\frac 34)e^{\frac{z^2}{4}}}{\sqrt{\pi}}\times[U(a,z)+U(a,-z)]\,;\label{C.Kummer_Para1} \\
	\Phi(a/2+3/4,3/2,z^2/2)=&\frac{2^{\frac a2-\frac 54}\Gamma(\frac a2+\frac 14)e^{\frac{z^2}{4}}}{\sqrt{\pi}z}\times[U(a,-z)-U(a,z)]\,.\label{C.Kummer_Para2}
\end{align}
Recall the integral representation of the parabolic cylinder function $U(a,z)$ by 12.5.1 in \cite{NIST}
\begin{equation}\label{C.Para}
	U(a,z)\,=\,\frac{\exp(-\frac{z^2}{4})}{\Gamma(\frac 12+a)}\int_{0}^{\infty} w^{a-\frac 12} \exp(-w^2/2-zw)dw\,,\qquad ~ Re(a)>-\frac 12\,.
\end{equation}
Lastly, the Poincar\'e-type asymptotic forms of confluent hypergeometric function hold (see 13.2.23 in \cite{NIST}):
\begin{equation}\label{C.Kummer_Asymp}
	\mathbb{M}(a,b,z)=\frac{1}{\Gamma(b)}\Phi(a,b,z)\asymp \frac{z^{a-b}}{\Gamma(a)}\exp(z)\,,\qquad\text{as }z\to \infty\,.
\end{equation}


\section*{Funding}
M. Li is supported by
the Fundamental Research Funds for the Central Universities,
China University of Geosciences (Wuhan, Grant number : CUG2106127 and CUGST2), China.
C. Huang is supported by the National Science Foundation of China, No. 11771163 and 12011530058.
Y. Hu is supported by    Natural Sciences and Engineering Research Council of Canada
  discovery fund and by a startup fund of University of Alberta.

\bibliographystyle{plain}
\bibliography{MS_stability_sep}


\end{document}